\documentclass[12pt]{amsart}
\usepackage{graphicx}
\usepackage{amsmath}
\usepackage{amsthm}
\usepackage{amsfonts}
\usepackage{amssymb}
\usepackage{color}
\usepackage{verbatim}
\usepackage{cite}
\usepackage{tikz}
\usepackage[pagewise]{lineno}

\theoremstyle{plain}
\newtheorem{thm}{Theorem}[section]
\newtheorem{lemma}[thm]{Lemma}
\newtheorem{cor}[thm]{Corollary}
\newtheorem{prop}[thm]{Proposition}

\newtheorem*{mainthm}{Theorem~\ref{Thm:EntropyConjugacy}}
\newtheorem*{realizationthm}{Theorem~\ref{Thm:Realization}}
\newtheorem*{thm*}{Theorem}

\theoremstyle{remark}
\newtheorem{rmk}[thm]{Remark}
\newtheorem{example}[thm]{Example}

\theoremstyle{definition}
\newtheorem{defn}[thm]{Definition}

\numberwithin{equation}{section}

\def\N{\mathbb{N}}
\def\Z{\mathbb{Z}}

\def\C{\mathcal{C}}

\DeclareMathOperator{\Pre}{Pre}

\DeclareMathOperator{\htop}{h_{top}}
\DeclareMathOperator{\hprob}{h_{prob}}

\newenvironment{PfofThmEntropyConjugacy}[1]
{\par\vskip2\parsep\noindent{\sc Proof of Theorem\ \ref{Thm:EntropyConj}. }}{{\hfill
$\Box$}
\par\vskip2\parsep}

\newenvironment{PfofPropPhiPreservesEntropy}[1]
{\par\vskip2\parsep\noindent{\sc Proof of Proposition\ \ref{Prop:PhiPreservesEntropy}. }}{{\hfill
$\Box$}
\par\vskip2\parsep}

\author[J. P. Kelly]{James P. Kelly}
\address[J. P. Kelly]{Department of Mathematics, Christopher Newport University, Newport News, VA 23606, USA}
\email{james.kelly@cnu.edu}
\title{Entropy Conjugacy for Markov Multi-maps of the Interval}
\author[K. McGoff]{Kevin McGoff}
\address[K. McGoff]{Department of Mathematics and Statistics, University of North Carolina at Charlotte, Charlotte, NC 28223, USA}
\email{kmcgoff1@uncc.edu}

\begin{document}

\begin{abstract}
We consider a class $\mathcal{F}$ of Markov multi-maps on the unit interval. Any multi-map gives rise to a space of trajectories, which is a closed, shift-invariant subset of $[0,1]^{\mathbb{Z}_+}$. For a multi-map in $\mathcal{F}$, we show that the space of trajectories is (Borel) entropy conjugate to an associated shift of finite type.  Additionally, we characterize the set of numbers that can be obtained as the topological entropy of a multi-map in $\mathcal{F}$.
\end{abstract}

\maketitle

\section{Introduction}

Multi-maps, also called set-valued maps, have been studied in the topological dynamics literature for some time, with such notable examples as \cite{Akin1993,McGehee1992,MillerAkin1999}. In the past decade multi-maps have been studied extensively, with a particular focus on the topological structure of the associated space of trajectories or a related inverse limit space; see \cite{Ingram2012}. This development has also led to a renewed interest in the dynamics of multi-maps \cite{KellyTennant2017,KennedyNall2018,ErcegKennedy2018,CordeiroPacifico2016}.
Additionally, multi-maps are the topological analogues of random maps of the interval, which have received substantial attention, e.g., \cite{buzzi1999exponential,froyland1999ulam,pelikan1984invariant,Quas}. 

In the study of single-valued maps of the interval, Markov maps \cite{bowen1979invariant} are particularly well-understood. These maps have a finite invariant set such that the map is strictly monotone on the intervals between elements of that set. This structure allows one to associate to each Markov interval map a corresponding shift of finite type that preserves many aspects of the dynamics. 

Recent work \cite{BanicCrepnjak2018,BanicLunder2016,CrepnjakLunder2016,AlvinKelly2018} has generalized the notion of Markov interval maps to the setting of multi-maps and established some of their basic properties.  
In particular, \cite{AlvinKelly2019} proves that under some conditions on the Markov multi-map, one may find upper and lower bounds for its entropy using associated shifts of finite type. 

Our main results substantially sharpen this previous work. Under mild conditions on the Markov multi-map, we associate to it a single shift of finite type, and then we establish a close connection (in the form of a Borel entropy conjugacy) between the dynamics of the multi-map and its associated shift of finite type. In particular, for a Markov multi-map in the class $\mathcal{F}$ considered here, the topological entropies of the Markov multi-map and of the associated shift of finite type must be equal. Furthermore, we demonstrate the richness of the class $\mathcal{F}$ by showing that any number that appears as the entropy of a shift of finite type also appears as the entropy of a Markov multi-map in $\mathcal{F}$.

\subsection{Statement of main results}

A multi-map of the unit interval is a function $F : [0,1] \to 2^{[0,1]}$, where $2^{[0,1]}$ is taken to be the set of closed subsets of $[0,1]$. Given such a multi-map, its trajectory space $X = X(F)$ is defined by
\[
X(F) = \Bigl\{ x = (x_n)_{n=0}^{\infty} \in [0,1]^{\mathbb{Z}_+} : \forall n \geq 0, x_{n+1} \in F(x_n) \Bigr\}.
\]
Further, let $\sigma_X : X \to X$ denote the left-shift map $(x_n)_{n=0}^{\infty} \mapsto (x_{n+1})_{n=0}^{\infty}$ on $X$. We seek to understand the multi-map $F$ by studying the dynamics of the system $(X,\sigma_X)$.  

In Section \ref{Sect:MarkovMultiMaps} we introduce a class of multi-maps that we call Markov multi-maps, and in Definition \ref{Defn:ProperParametrization} we state what it means for a Markov multi-map to be properly parametrized. 
In this work we focus on a specific class $\mathcal{F}$ of Markov multi-maps (see Definition \ref{Defn:ClassF}): properly parametrized Markov multi-maps with complete sets of coding and avoiding words and positive entropy. 
For any multi-map $F$ in this class, one may associate to $F$ a square matrix $M = M(F)$ with entries in $\{0,1\}$ (see Section \ref{Sect:Matrix}). The matrix $M$ encodes the combinatorial structure of $F$. Let $\Sigma_M$ be the shift of finite type defined by $M$, with left-shift map $\sigma_M$. The following theorem provides a precise correspondence between a ``large" subset of the trajectory space $X$ and a ``large" subset of the SFT $\Sigma_M$, where ``large" here refers to a notion of entropy. A precise definition of Borel entropy conjugacy, originally defined by Buzzi \cite{Buzzi1997} under the term ``entropy conjugacy," appears in Definition \ref{Defn:EntropyConj}.

\begin{thm} \label{Thm:EntropyConj}
Let $\mathcal{F}$ be the class of Markov multi-maps specified in Definition \ref{Defn:ClassF}.
Let $F$ be in $\mathcal{F}$ with trajectory space $X$ and associated SFT $\Sigma_M$. Then $(X,\sigma_X)$ is Borel entropy conjugate to $(\Sigma_M,\sigma_M)$. 
\end{thm}

Since Borel entropy conjugacy is known to preserve topological entropy, we immediately obtain the following corollary.

\begin{cor}
Let $F$ be in $\mathcal{F}$ with trajectory space $X$ and associated SFT $\Sigma_M$. Then $\htop(X,\sigma_X) = \htop(\Sigma_M,\sigma_M)$. 
\end{cor}

In fact, since Borel entropy conjugacy provides a correspondence between all ergodic measures with large enough entropy, the following corollary is also immediate.

\begin{cor}
Let $F$ be in $\mathcal{F}$ with trajectory space $X$ and associated SFT $\Sigma_M$. Then $(X,\sigma_X)$ has the same number of measures of maximal entropy as $(\Sigma_M,\sigma_M)$. In particular, if $(\Sigma_M,\sigma_M)$ is irreducible, then $(X,\sigma_X)$ is intrinsically ergodic (i.e., has a unique measure of maximal entropy).
\end{cor}

\begin{rmk} Random maps of the interval have been studied primarily with an eye towards the existence and properties of absolutely continuous invariant measures, e.g., see \cite{bowen1979invariant,Quas}. While the entropy conjugacy guaranteed by Theorem \ref{Thm:EntropyConj} provides a correspondence between ergodic measures of large entropy on $X$ and on $\Sigma_M$, it does not address questions about the whether any of these measures is absolutely continuous on $X$.
\end{rmk}

Let us now answer a question of Karl Petersen (personal communication). Let $\mathcal{H}(\mathcal{F})$ denote the set of real numbers $r >0$ such that there exists a multi-map $F \in \mathcal{F}$ having trajectory space $X$ with $\htop(X,\sigma_X) = r$. Recall that Lind has characterized the set of positive real numbers that arise as the entropy of a SFT as the set of all positive rational multiples of logarithms of Perron numbers \cite{Lind1984}. 

\begin{thm}\label{Thm:Realization}
The set $\mathcal{H}(\mathcal{F})$ is equal to the set of all positive rational multiples of logarithms of Perron numbers.
\end{thm}

\subsection{Organization of the paper}

In Section \ref{Sect:Background}, we provide background information and notation concerning shifts of finite type, ergodic theory, and Borel entropy conjugacy. Section \ref{Sect:MarkovMultiMaps} introduces Markov multi-maps and the class $\mathcal{F}$ of interest. Taken together, Sections \ref{Sect:PreliminaryResults} -- \ref{Sect:MainProof} contain the proof of our main result, Theorem \ref{Thm:EntropyConj}. In Section \ref{Sect:Sufficient} we establish some sufficient conditions for a Markov multi-map to be in $\mathcal{F}$, and then in Section \ref{Sect:Realization} we prove the realization result, Theorem \ref{Thm:Realization}. Finally, Section \ref{Section:Examples} contains some examples of Markov multi-maps.

\section{Background and notation} \label{Sect:Background}

We denote by $2^{[0,1]}$ the set of all non-empty, closed subsets of $[0,1]$. A \emph{multi-map} on $[0,1]$ is a function $F\colon[0,1]\to2^{[0,1]}$. The \emph{graph} of a multi-map $F$ is the set $G(F)=\{(x,y)\in[0,1]^2\colon y\in F(x)\}$.  
A \emph{trajectory} for $F$ is a sequence $(x_0,x_1,\ldots)\in[0,1]^{\Z_+}$ such that for all $n\geq 1$, we have $x_n\in F(x_{n-1})$, or equivalently $(x_{n-1},x_n)\in G(F)$. We denote by $X = X(F)$ the set of trajectories for $F$, and we give $X$ the topology it inherits as a subspace of $[0,1]^{\Z_+}$ with the product topology. We also define the \emph{left-shift} on $X$, denoted $\sigma_X$, by setting $\sigma_X(x_0,x_1,\ldots)=(x_1,x_2,\ldots)$. Observe that $\sigma_X$ is a continuous 
mapping on $X$, and if $G(F)$ is closed in $[0,1]^2$, then $X$ is closed in $[0,1]^{\Z_+}$.

\subsection{Shifts of finite type}

Let $\mathcal{A}$ be a finite set, which we call the \emph{alphabet}. An element $b\in\mathcal{A}^n$ is called a \emph{word} of length $n$. The full shift on $\mathcal{A}$ is $\Sigma=\mathcal{A}^{\Z_+}$, endowed with the product topology induced by the discrete topology on $\mathcal{A}$. Given a set of words $\mathcal{F}$, we may define $\Sigma_\mathcal{F}\subseteq\Sigma$ to be the set of points that do not contain any word in $\mathcal{F}$. We refer to words in $\mathcal{F}$ as \emph{forbidden words}. Then $\Sigma_\mathcal{F}$ is closed and invariant under the left-shift on $\Sigma$. If $\mathcal{F}$ is finite, then we refer to $\Sigma_\mathcal{F}$ as a \emph{shift of finite type (SFT)}. In this work we restrict attention to SFTs for which all the forbidden words have length two, called nearest neighbor SFTs. For more on SFTs, we refer the reader to the book \cite{Lind1995}.

Any nearest neighbor SFT may be expressed in terms of a directed graph, $(V,E)$, where the set of vertices $V$ is equal to $\mathcal{A}$, and given $a,b\in\mathcal{A}$, there is an edge from $a$ to $b$ in the edge set $E$ if and only if $ab\notin\mathcal{F}$. Furthermore, we associate to any such graph its \emph{adjacency matrix} $M$, defined  as the square matrix indexed by $\mathcal{A}$ such that for $a,b\in\mathcal{A}$, if $ab\notin\mathcal{F}$ then $M(a,b) = 1$, and otherwise $M(a,b)=0$. Note that any zero-one matrix indexed by $\mathcal{A}$ also defines an associated nearest neighbor SFT (by letting $ab$ be a forbidden word whenever $M(a,b) = 0$). The nearest neighbor SFT defined by a zero-one matrix $M$ is denoted by $\Sigma_M$, and the left-shift restricted to $\Sigma_M$ is denoted by $\sigma_M$.

In what follows it is convenient to have some notation for words of arbitrary length that do not contain any forbidden word. For $n \geq 1$, we let $\mathcal{L}_n$ denote the set of words $a_0 \dots a_n \in \mathcal{A}^{n+1}$ such that $M(a_i,a_{i+1}) = 1$ for each $i = 0, \dots, n-1$. Then let
\[
\mathcal{L} = \bigcup_{n \geq 0} \mathcal{L}_n,
\]
where $\mathcal{L}_0 = \mathcal{A}$.

A nearest neighbor SFT defined by the matrix $M$ is \emph{irreducible} if for every pair of non-empty, open sets $U,V\subseteq\Sigma_M$, there exists $n\geq 1$ such that $\sigma_M^n(U)\cap V\neq \emptyset$. Equivalently, $\Sigma_M$ is irreducible if for each $a,b \in \mathcal{A}$, there exists $n \geq 1$ such that $M^n(a,b) >0$. 

Consider an arbitrary nearest neighbor SFT $\Sigma_M$ on alphabet $\mathcal{A}$. It has an associated finite directed graph $\Gamma$, with vertex set $\mathcal{A}$ and an edge from $a$ to $b$ whenever $M(a,b) = 1$. Let $\mathcal{C}_1,\dots,\mathcal{C}_K \subset \mathcal{A}$ be the vertex sets of the maximal strongly connected components of $\Gamma$, which we call the irreducible components of $\Gamma$. For each $\mathcal{C}_k$, the set of points in $\Sigma_M$ containing only symbols from $\mathcal{C}_k$ forms an irreducible SFT, which we denote by $\Sigma_M(\mathcal{C}_k)$. We refer to $\Sigma_M(\mathcal{C}_k)$ as an irreducible component of $\Sigma_M$. Note that the irreducible components $\Sigma_M(\mathcal{C}_1), \dots, \Sigma_M(\mathcal{C}_k)$ are pairwise disjoint, and the set $\Sigma_{M} \setminus \bigcup_k \Sigma_M(\mathcal{C}_k)$ contains only wandering points. See \cite[Chapter 4]{Lind1995} for more details on this decomposition. We also denote by $\mathcal{L}(\mathcal{C}_k)$ the set of words of arbitrary length on $\mathcal{C}_k$ that do not contain a forbidden words. 


\subsection{Invariant measures and entropy}

In this work a topological dynamical system consists of a pair $(\mathcal{X},T)$, where $T : \mathcal{X} \to \mathcal{X}$ is a continuous self-map of a compact metrizable space. For any such system, we let $\mathcal{M}(\mathcal{X},T)$ denote the set of Borel probability measures $\mu$ on $\mathcal{X}$ such that $\mu(E) = \mu(T^{-1} E)$ for all Borel sets $E \subset \mathcal{X}$. Note that $\mathcal{M}(\mathcal{X},T)$ is a nonempty, convex set that is compact in the weak$^*$ topology. A measure $\mu \in \mathcal{M}(\mathcal{X},T)$ is called ergodic if $\mu(E) \in \{0,1\}$ for all Borel sets $E$ such that $T^{-1}(E) \subset E$. The set of ergodic measures is denoted by $\mathcal{M}_e(\mathcal{X},T)$. Note that a measure $\mu \in \mathcal{M}(\mathcal{X},T)$ is an extreme point in $\mathcal{M}(\mathcal{X},T)$ if and only if $\mu$ is ergodic.

The following notation is used in subsequent sections. For any Borel set $E \subset \mathcal{X}$, the union of all of its pre-images is denoted
\[
\Pre(E) = \bigcup_{ n \geq 0} T^{-n}(E).
\]
Note that $T^{-1}( \Pre(E)) \subset \Pre(E)$, and therefore if $\mu \in \mathcal{M}_e(\mathcal{X},T)$ then $\mu(\Pre(E)) \in \{0,1\}$.

We also require some elementary facts regarding the entropy theory of dynamical systems. Complete definitions and proofs can be found in \cite{Walters2000}. Let $\htop(T)$ denote the topological entropy of the topological system $(\mathcal{X},T)$. Furthermore, when the system $(\mathcal{X},T)$ is understood and $\mu \in \mathcal{M}(\mathcal{X},T)$, we denote the measure-theoretic entropy of $\mu$ by $h(\mu)$. The standard variational principle for entropy states that 
\[
\htop(T) = \sup_{ \mu \in \mathcal{M}(\mathcal{X},T)} h(\mu),
\]
and the supremum may be taken over only the ergodic measures. Furthermore, for SFTs it is known that the supremum is achieved, and if the SFT is irreducible, then it has a unique measure of maximal entropy.
Furthermore, we note for future use that an irreducible SFT is entropy minimal, i.e., if $X$ is an irreducible SFT of positive entropy and $Y$ is a strict subset of $X$, then $\htop(Y,\sigma|_Y) < \htop(X,\sigma|_X)$ (see \cite{Lind1995} for a proof).

\subsection{Entropy conjugacy}

We adopt the following definition of entropy for Borel sets (following Buzzi \cite{Buzzi1997}).
\begin{defn}
	Let $T : \mathcal{X} \to \mathcal{X}$ be a topological dynamical system. For a Borel set $E \subset \mathcal{X}$, let 
	\[
	\hprob(E) = \sup \bigl\{ h(\mu) : \mu \in \mathcal{M}_e(\mathcal{X},T), \, \mu(E) >0 \bigr\}.
	\]
\end{defn}

Now we define a notion of entropy conjugacy, which was previously introduced by Buzzi \cite{Buzzi1997}.
\begin{defn} \label{Defn:EntropyConj}
	Suppose that $T_0 : \mathcal{X}_0 \to \mathcal{X}_0$ and $T_1 : \mathcal{X}_1 \to \mathcal{X}_1$ are topological dynamical systems. We say that they are \textit{Borel entropy conjugate} if there exist Borel sets $E_0 \subset \mathcal{X}_0$ and $E_1 \subset \mathcal{X}_1$ and an invertible Borel bi-measurable map $\psi : \mathcal{X}_0 \setminus E_0 \to \mathcal{X}_1 \setminus E_1$ such that
	\begin{itemize}
		\item $\hprob(E_0) < \htop(\mathcal{X}_0,T_0)$;
		\item $\hprob(E_1) < \htop(\mathcal{X}_1,T_1)$; and
		\item $\psi \circ T_0 = T_1 \circ \psi$ on $\mathcal{X}_0 \setminus E_0$.
	\end{itemize}
\end{defn}

It is an easy corollary of the variational principle for topological dynamical systems that if $(\mathcal{X}_0,T_0)$  and $(\mathcal{X}_1,T_1)$ are Borel entropy conjugate, then $\htop(\mathcal{X}_0,T_0) = \htop(\mathcal{X}_1,T_1)$.

\begin{rmk}
In his work on topological entropy for non-compact sets, Bowen introduced a notion that he called entropy conjugacy \cite{bowen1973topological}. Bowen's definition of entropy conjugacy requires that the sets $E_0$ and $E_1$ have smaller topological entropy (in the dimension-theoretic sense defined in his paper) than the full system and that the conjugating map $\psi$ is continuous. As such, Bowen's notion of entropy conjugacy is stronger than the notion of Borel entropy conjugacy defined above.
\end{rmk}

\section{Markov multi-maps} \label{Sect:MarkovMultiMaps}

We now give a precise definition of Markov multi-maps on the interval $[0,1]$. This definition is based on the one given in \cite{AlvinKelly2019}, though our definition is slightly less general.

\begin{defn}
A Markov multi-map $F$ of the interval $[0,1]$ is defined by a tuple $(P, \mathcal{A}_0, \mathcal{A}_1,\mathcal{A}_2, D, R, \{ f_a \}_{a \in \mathcal{A}_0})$ satisfying the following conditions:
\begin{enumerate}
 \item $P = \{p_0,\dots,p_r\}$ is a partition of the interval $[0,1]$ with $0 = p_0 < \dots < p_r = 1$;
 \item $\mathcal{A} = \mathcal{A}_0 \sqcup \mathcal{A}_1 \sqcup \mathcal{A}_2$ is a finite set;
 \item $D : \mathcal{A} \to 2^{[0,1]}$, and for each $a \in \mathcal{A}$, there exists $p_i \in P$ such that 
 \[
 D(a) = \left\{ \begin{array}{ll}
                     [p_i,p_{i+1}], & \text{if }  a \in \mathcal{A}_0 \\
                      \{p_i\}, & \text{if } a \in \mathcal{A}_1 \cup \mathcal{A}_2;
                    \end{array}
                    \right.
 \]
 \item $R : \mathcal{A} \to 2^{[0,1]}$, and for each $a \in \mathcal{A}$, there exists $u \leq v$ in $P$ such that $R(a) = [u,v]$ and
 \[
 \left\{ \begin{array}{ll}
             u < v, & \text{if } a \in \mathcal{A}_0 \\
             u < v \text{ and } R(a) \cap P = \{u,v\}, & \text{if } a \in \mathcal{A}_1 \\
             u = v, & \text{if } a \in \mathcal{A}_2;
          \end{array}
          \right.
 \]
   \item for each $a \in \mathcal{A}_0$, the map $f_a : D(a) \to R(a)$ is a homeomorphism;
 \item $[0,1] \subset \bigcup_{a \in \mathcal{A}} D(a)$.
\end{enumerate}
\end{defn}


\subsection{The graph of a Markov multi-map}

Let $F$ be a Markov multi-map.
For $a \in \mathcal{A}_0$, let $G(a)$ denote the graph of $f_a$. 
For $a \in \mathcal{A}_1 \cup \mathcal{A}_2$, let $G(a) = D(a) \times R(a)$. 
Then the graph of $F$ is
\[
G(F) = \bigcup_{a \in \mathcal{A}} G(a).
\]
Note that each $G(a)$ is closed in $[0,1] \times [0,1]$, and so is $G(F)$.
Some examples of Markov multi-maps and their graphs are given in Section~\ref{Section:Examples}.

Now we make some additional graph-related definitions that are used repeatedly throughout this work.
\begin{defn}
Let $a \in \mathcal{A}$.
\begin{itemize}
\item Suppose $a \in \mathcal{A}_0$ with $D(a) = [p_i,p_{i+1}]$ and $R(a) = [u,v]$. Define $D_0(a) = (p_i,p_{i+1})$ and $R_0(a) = (u,v)$, and let $G_0(a)$ be the graph of $f_a|_{D_0(a)}$.

\vspace{2mm}

\item Suppose $a \in \mathcal{A}_1$ with $D(a) = \{p\}$ and $R(a) = [u,v]$. Define $D_0(a) = \{p\}$ and $R_0(a) = (u,v)$, and let $G_0(a) = \{p\} \times R_0(a)$.

\vspace{2mm}

\item Suppose $a \in \mathcal{A}_2$ with $D(a) = \{p\}$ and $R(a) = \{q\}$. Define $D_0(a) = \{p\}$ and $R_0(a) = \{q\}$, and let $G_0(a) = \{(p,q)\}$.
\end{itemize}
\end{defn}

Our results require that $F$ has some additional structure, which we now begin to define.
\begin{defn}
We say that $F$ satisfies the \textit{no crossing property} if the following holds: for all $a,b \in \mathcal{A}_0$, if $G_0(a) \cap G_0(b) \neq \varnothing$ then $a = b$.
\end{defn}

The following property strictly implies the no crossing property.
\begin{defn} \label{Defn:ProperParametrization}
We say that $F$ is \textit{properly parametrized} if the collection $\{G_0(a) : a \in \mathcal{A} \}$ forms a partition of $G(F)$. 
\end{defn} 

We think of the no crossing property as a property of the graph $G(F)$ (and the partition $P$), whereas being properly parametrized depends on the particular parametrization of the Markov multi-map $F$. However, these properties are related by Lemma \ref{Lemma:ProperParam}:  if $F_0$ is a Markov multi-map with the no crossing property, then there exists a properly parametrized Markov multi-map $F_1$ such that $G(F_0) = G(F_1)$. 

\begin{rmk} 
If $F$ is a Markov multi-map, then it possesses the following graph \textit{Markov property}: for all $a,b \in \mathcal{A}$, if $D_0(b) \cap R_0(a) \neq \varnothing$, then $D_0(b) \subset R_0(a)$. This property is used to define the SFT associated with $F$, which appears in the next section.
\end{rmk}

\subsection{The SFT associated to a Markov multi-map} \label{Sect:Matrix}

We associate to any Markov multi-map $F$ an SFT as follows.
Let $M$ be the square matrix indexed by $\mathcal{A}$ such that for $a,b \in \mathcal{A}$,
\[
M(a,b) = \left\{ \begin{array}{ll} 
                           1, & \text{if } D_0(b) \subset R_0(a) \\
                           0, & \text{otherwise}.
                       \end{array}
                       \right.
\]
Let $\Sigma_M \subset \mathcal{A}^{\mathbb{Z}_+}$ be the nearest neighbor SFT with alphabet $\mathcal{A}$ and transition matrix $M$.
%
%
%
%

In our main results, we relate the SFT $\Sigma_M$ to the trajectory space $X$. In particular, Theorem~\ref{Thm:EntropyConj} establishes sufficient conditions for these systems to be Borel entropy conjugate.

%

\subsection{Nested intervals}

Let $F$ be a properly parametrized Markov multi-map with associated matrix $M$.
Here we associate to each sequence $\mathbf{a} \in \Sigma_M$ a nonempty closed (possibly degenerate) interval in $[0,1]$.
To begin, for each $a \in \mathcal{A}_0$, we let $f_a^{-1}$ be the standard inverse function (which exists since $f_a$ is assumed to be a homeomorphism).
For $a \in \mathcal{A}_1 \cup \mathcal{A}_2$, we let $f_a^{-1}$ be the unique map such that $f_a^{-1} : R(a) \to D(a)$ (which exists since $R(a)$ is non-empty and $D(a)$ is a singleton in this case).

Let $u = a_0 \dots a_n \in \mathcal{L}_n$.
Define the set
\[
I_{u} = f_{a_0}^{-1} \circ \dots \circ f_{a_{n-1}}^{-1}( D(a_n))
\]
We make the following elementary observations.
\begin{itemize}
\item $I_u$ is non-empty. (Since $M(a_i,a_{i+1}) = 1$, we have that $D(a_{i+1}) \subset R(a_i)$, so $f_{a_i}^{-1}$ maps $D(a_{i+1})$ into $D(a_i)$.)
\item $I_u$ is a closed (possibly degenerate) interval, since $I_u$ is the image of the closed interval $D(a_n)$ under the continuous, monotone map $f_{a_0}^{-1} \circ \dots \circ f_{a_{n-1}}^{-1}$.
\item $I_{a_0 \dots a_{n+1}} = f_{a_0}^{-1} (I_{a_1 \dots a_{n+1}})$.
\item $I_{a_0 \dots a_{n+1}} \subset I_{a_0 \dots a_n}$ (since $f_{a_n}^{-1}(D(a_{n+1})) \subset D(a_n)$).
\end{itemize}

Now consider $\mathbf{a} = (a_n)_{n=0}^{\infty} \in \Sigma_M$. Then $\{I_{a_0 \dots a_n}\}_{n=1}^{\infty}$ is a nested sequence of non-empty, closed intervals in $[0,1]$. Let
\[
I_{\mathbf{a}} = \bigcap_{n=1}^{\infty} I_{a_0 \dots a_n}.
\]
Then $I_{\mathbf{a}}$ is a non-empty, closed interval. Additionally, we note that
\[
I_{\mathbf{a}} = f_{a_0}^{-1}( I_{\sigma(\mathbf{a})}).
\]
These intervals appear in the next section in the definitions that characterize the class of Markov multi-maps in our main results.

\subsection{Definition of the class $\mathcal{F}$} \label{Sect:ClassF}

In this section we define the class $\mathcal{F}$ of Markov multi-maps that appears in our main results.  Let $F$ be a properly parametrized Markov multi-map with associated matrix $M$.

\begin{defn}
Suppose $\mathcal{C} \subset \mathcal{A}$ is an irreducible component of the graph with adjacency matrix $M$. We say that $\mathcal{C}$ has a \textit{coding word} if there exists $u \in \mathcal{L}(\mathcal{C})$ such that if $\mathbf{a} \in \Sigma_M(\mathcal{C})$ and $\{ n \geq 0 : \sigma^n(\mathbf{a}) \in [u] \}$ is infinite, then $I_{\mathbf{a}}$ is a singleton. Furthermore, we say that $F$ has a \textit{complete set of coding words} if each irreducible component with positive entropy has a coding word.
\end{defn}


\begin{defn}
Suppose $\mathcal{C} \subset \mathcal{A}$ is an irreducible component of the graph with adjacency matrix $M$. We say that $\mathcal{C}$ has an \textit{avoiding word} if there exists $u \in \mathcal{L}(\mathcal{C})$ such that if $\mathbf{a} \in [u]$, then $I_{\mathbf{a}} \cap P = \varnothing$. Furthermore, we say that $F$ has a \textit{complete set of avoiding words} if the following condition holds: if $\mathcal{C}$ is an irreducible component with positive entropy that is entirely contained in $\mathcal{A}_0$, then $\mathcal{C}$ has an avoiding word.
\end{defn}

Now we are prepared to give a precise definition of the class of Markov multi-maps that appears in our main results.

\begin{defn} \label{Defn:ClassF}
The class $\mathcal{F}$ consists of all properly parametrized Markov multi-maps $F$ such that $F$ has a complete set of coding words, $F$ has a complete set of avoiding words, and the associated SFT $\Sigma_M$ has positive entropy. 
\end{defn}

\subsection{Finite labeled trajectories}

Here we define some additional terminology that is useful in the following sections.

\begin{defn}
Let $F \in \mathcal{F}$. Let $m \geq 1$.
We say that  $x = x_0,\dots,x_m \in [0,1]^{m+1}$ is a \textit{finite trajectory} of $F$ if 
\[
(x_n,x_{n+1}) \in G(F), \quad \forall n = 0,\dots,m-1.
\]

Next, we say that $(x,b) \in [0,1]^{m+1} \times \mathcal{L}_{m-1}$ is a \textit{finite labeled trajectory} (of length $m+1$) if 
\[
(x_n,x_{n+1}) \in G(b_n), \quad \forall n = 0,\dots,m-1.
\]
Let $\mathcal{T}_{m}$ be the set of finite labeled trajectories of length $m+1$. We endow $\mathcal{T}_m$ with the subspace topology inherited from $[0,1]^{m+1} \times \mathcal{L}_{m-1}$ (which has the product of the usual topology on $[0,1]^{m+1}$ and the discrete topology on $\mathcal{L}_{m-1}$).

Finally, we say that $(x,b) \in \mathcal{T}_m$ is a \textit{special finite labeled trajectory} of length $m+1$ if
\[
(x_n,x_{n+1}) \in G_0(b_n), \quad \forall n = 0,\dots,m-1.
\]
Let $\mathcal{S}_m$ denote the set of special finite labeled trajectories of length $m+1$, and we let $\mathcal{S}_m$ inherit the subspace topology inherited from $\mathcal{T}_m$.
\end{defn}

\begin{rmk}
Let $F$ be in $\mathcal{F}$, and let $x$ be a finite trajectory of $F$ of length $m+1$. 
Since $G(F)$ is the union of the sets $\{G(a)\}_{a \in \mathcal{A}}$, there exists $b \in \mathcal{L}_{m-1}$ such that $(x,b)$ is in $\mathcal{T}_m$. In fact, since $F$ is properly parametrized, the sets $\{G_0(a)\}_{a \in \mathcal{A}}$ form a partition of $G(F)$, and therefore there exists a \textit{unique} element $b \in \mathcal{L}_{m-1}$ such that $(x,b)$ is in $\mathcal{S}_m$. 
\end{rmk}

\section{Preliminary results} \label{Sect:PreliminaryResults}

\subsection{Parametrization lemma}

The following simple result states that any Markov multi-map with the no-crossing property can be properly parametrized without changing its graph. Since the space of trajectories of a Markov multi-map depends only on its graph, this reparametrization also preserves the space of trajectories.

\begin{lemma} \label{Lemma:ProperParam}
Suppose that $F_0$ is a Markov multi-map with the no-crossing property. Then there exists a properly parametrized Markov multi-map $F_1$ with $G(F_0) = G(F_1)$.
\end{lemma}
\begin{proof}
Let $F$ be a Markov multi-map with the no-crossing property. Let $\mathcal{B}_0 = \mathcal{A}_0$. Let
\[
\mathcal{B}_1 = \biggl\{ \{p_i\} \times [p_j,p_{j+1}] : p_i,p_j \in P \text{ and } \exists a \in \mathcal{A}_1, \{p_i\} \times [p_j,p_{j+1}] \subset G(a) \biggr\}.
\]
Let 
\[
\mathcal{B}_2 = \biggl\{ (p,q) : p,q \in P, (p,q) \in G(F) \biggr\}.
\]
Then let $F_1$ be the Markov multi-map defined by $\mathcal{B}_0$, $\mathcal{B}_1$, and $\mathcal{B}_2$.
\end{proof}

\subsection{Graph lemmas}

In this section we prove a few facts about graphs of Markov multi-maps. 
Throughout the remainder of this section, we consider $F \in \mathcal{F}$. 
Since $F$ is properly parametrized, we know that if $a,b\in\mathcal{A}$ are distinct elements, then $G_0(a)\cap G_0(b)=\varnothing$. However, it is possible that $a \neq b$ and yet $G(a)$ has nontrivial intersection with $G(b)$. The following lemma shows that any such intersection must be contained in $P\times P$.

\begin{lemma} \label{Lemma:Marvin}
Let $F$ be in  $\mathcal{F}$.
Suppose $(x,y) \in G(F) \setminus (P \times P)$. Then there is a unique $a \in \mathcal{A}$ such that $(x,y) \in G(a)$, and furthermore $(x,y) \in G_0(a)$.
\end{lemma}
\begin{proof}
Since $(x,y) \in G(F) = \cup_a G(a)$, there must exist some $a \in \mathcal{A}$ such that $(x,y) \in G(a)$. For uniqueness, suppose that $(x,y) \in G(a) \cap G(b)$.  By the no-crossing property, for any $a \neq b$, we have $G(a) \cap G(b) \subset P \times P$. Since $(x,y) \notin P \times P$, we conclude that $a = b$. 

Since $(x,y) \notin P \times P$ and $(x,y) \in G(a)$, we see that $a \in \mathcal{A}_0 \cup \mathcal{A}_1$, and we must have $(x,y) \in G_0(a)$.
\end{proof}

The next lemma asserts that $G(F)$ cannot accumulate along a horizontal line to any point of $G(F) \cap (P \times P)$. 

\begin{lemma} \label{Lemma:Lennon}
Let $(p,q) \in G(F) \cap (P \times P)$. Then there exists an open set $U \subset [0,1] \times [0,1]$ such that $(p,q) \in U$ and if $(y,q) \in G_0(a) \cap U$, then $y = p$ and $a$ is the unique element of $\mathcal{A}_2$ such that $G(a) = \{(p,q)\}$.
\end{lemma}
\begin{proof}
Let $L_q = [0,1] \times \{q\}$. By our definition of Markov multi-map, $L_q \cap G(F)$ is a finite set containing $(p,q)$. Then there exists a relatively open interval $I$ in $[0,1]$ such that $I \times \{q\} \cap G(F) = \{(p,q)\}$. Let $U = I \times [0,1]$. Then $U$ is open in $[0,1]\times[0,1]$ and if $(y,q) \in G_0(a)$, then $y = p$ and $a$ must be the unique element of $\mathcal{A}_2$ such that $G(a) = \{(p,q)\}$.
\end{proof}

The next two lemmas address the convergence of sequences in the space of finite labeled trajectories.

\begin{lemma} \label{Lemma:Sam}
Let $m \geq 1$.
Suppose that $x = x_0,\dots,x_{m+1} \in [0,1]^{m+2}$ is a finite trajectory of $F$ such that  
\begin{itemize}
\item $x_0,\dots,x_{m} \in P$, and
\item $x_{m+1} \in [0,1] \setminus P$.
\end{itemize}
Let $w$ be the unique element of $\mathcal{L}_m$ such that $(x,w) \in \mathcal{S}_{m+1}$.
If the sequence $\{(y^k,b^k)\}_{k=1}^{\infty}$ is in $\mathcal{S}_{m+1}$ and converges to $(x,b)$ in $[0,1]^{m+2} \times \mathcal{A}^{m+1}$, then $b = w$. 
\end{lemma}
\begin{proof}
By our hypotheses on $x_0,\dots,x_{m+1}$, we have that $w_n \in \mathcal{A}_2$ for $n = 0,\dots,m-1$ and $w_m \in \mathcal{A}_1$. Let us now show that $b_0 \dots b_m = w$.

First, note that since $\mathcal{A}^{m+1}$ has the discrete topology, for all large enough $k$, we have $b^k = b$. Then for all large enough $k$, we have $(y^k_n,y^k_{n+1}) \in G(b_n)$ for all $n = 0,\dots,m$. Since $G(b_n)$ is closed and $\{(y^k_n,y^k_{n+1})\}_{k=1}^{\infty}$ converges to $(x_n,x_{n+1})$, we see that $(x_n,x_{n+1}) \in G(b_n)$ for each $n = 0 ,\dots, m$. 

Since $x_{m+1} \notin P$, Lemma \ref{Lemma:Marvin} gives that there is a unique $a \in \mathcal{A}$ such that $(x_m,x_{m+1}) \in G(a)$, and therefore we must have $b_m = a = w_m$. Furthermore, since $x_{m} \in P$ and $x_{m+1} \notin P$, we see that $w_m \in \mathcal{A}_1$, which implies that $D_0(w_m) = \{x_m\}$. Then since $(y^k_m,y^k_{m+1}) \in G_0(b^k_m) = G_0(b_m) = G_0(w_m)$ for all large enough $k$ and $D_0(w_m) = \{x_m\}$, we see that $y^k_m = x_m$ for all large enough $k$.

We claim by backwards induction that for each $j = 0,\dots,m$, we have $b_j = w_j$ and $y^k_j = x_j$ for all large enough $k$. We have established the base case ($j = m$) in the preceding paragraph. Now suppose it holds for some $j+1$. Let $U$ be given by Lemma \ref{Lemma:Lennon} for the point $(x_j,x_{j+1})$. By the inductive hypothesis, for all large enough $k$, we have $y^k_{j+1} = x_{j+1} \in P$. Also, for all large enough $n$, we must have $(y^k_j,y^k_{j+1}) \in U$ (since $\{(y^k_j,y^k_{j+1})\}_{k=1}^{\infty}$ converges to $(x_j,x_{j+1})$). Then for all large enough $k$, we have $(y^k_j, x_{j+1}) \in U \cap G_0(b_j)$. By our choice of $U$, we must have that $b_j = w_j$ and $y^k_j = x_j$ for all large enough $k$, which completes the induction.
\end{proof}

\begin{lemma} \label{Lemma:McCartney}
Let $m \geq 1$.
Suppose that $x = x_0,\dots,x_{m+1} \in [0,1]^{m+2}$ is a finite trajectory of $F$ such that  $x_{m+1} \in [0,1] \setminus P$.
Let $w$ be the unique element of $\mathcal{L}_m$ such that $(x,w) \in \mathcal{S}_{m+1}$.
If the sequence $\{(y^k,b^k)\}_{k=1}^{\infty}$ is in $\mathcal{S}_{m+1}$ and converges to $(x,b)$ in $[0,1]^{m+2} \times \mathcal{A}^{m+1}$, then $b = w$. 
\end{lemma}
\begin{proof}
As $\mathcal{A}^{m+1}$ has the discrete topology, we must have that $b^k = b$ for all large enough $k$. Then for all large enough $k$, we have $(y^k_n,y^k_{n+1}) \in G(b_n)$, which is closed, and therefore $(x_n,x_{n+1}) \in G(b_n)$.

Observe that if $(x_n,x_{n+1}) \in G(F) \setminus (P \times P)$, then $b_n = w_n$ by Lemma \ref{Lemma:Marvin}. Now suppose that we have some $n$ such that $(x_n,x_{n+1}) \in P \times P$. Then there exists $N \in [n+1, \, m]$ such that $x_j \in P$ for all $j = n, \dots, N$ and $x_{N+1} \notin P$. Thus $x_n, \dots, x_{N+1}$ satisfies the conditions of Lemma \ref{Lemma:Sam}, and we conclude that $b_n = w_n$.
\end{proof}

%

\section{Construction of the joint system and factor maps}

Let $F$ be a Markov multi-map in $\mathcal{F}$ with associated trajectory space $X$ and SFT $\Sigma_M$. In the following section, we introduce a topological dynamical system by taking limits of special finite labeled trajectories. We call this system the \textit{joint system}. Then in Sections \ref{Sect:FactoringOntoSigma} and \ref{Sect:FactoringOntoX}, we show that the joint system is in fact a common extension of $X$ and $\Sigma_M$. The joint system and its factor maps onto $X$ and $\Sigma_M$ are central to the construction of the Borel entropy conjugacy in our proof of Theorem \ref{Thm:EntropyConj}. We establish their key properties in this section. 

\subsection{The joint system}

Let $F$ be a Markov multi-map in $\mathcal{F}$ with associated SFT $\Sigma_M$. Here we define a subset $V$ of the product space $[0,1]^{\mathbb{Z}_+} \times \Sigma_M$, which will serve as a common extension of the trajectory space $X$ and the SFT $\Sigma_M$. 

\begin{defn} \label{Defn:V}
Let $F$ be in  $\mathcal{F}$ with associated trajectory space $X$ and SFT $\Sigma_M$. Define a set $V = V(F) \subset [0,1]^{\mathbb{Z}_+} \times \Sigma_M$ as follows.
A pair $(x,\mathbf{a}) \in [0,1]^{\mathbb{Z}_+} \times \Sigma_M$ is in $V$ if there exists a sequence $\{\ell_k\}_{k=1}^{\infty}$ of natural numbers tending to infinity and a sequence $\{(y^k,a^k)\}_{k=1}^{\infty}$ of special finite labeled trajectories, with $(y^k,a^k) \in \mathcal{S}_{\ell_k}$, such that for each $n \geq 0$, the sequence $\{(y^k_n,a^k_n)\}_{k=1}^{\infty}$ converges to $(x_n,a_n)$ in $[0,1] \times \mathcal{A}$.
%
\end{defn}

\begin{prop}
$V$ is closed and invariant under the left shift.
\end{prop}
\begin{proof}
Suppose that $(x^m,\mathbf{a}^m)$ is a sequence in $V$ that converges in $[0,1]^{\mathbb{Z}_+} \times \Sigma_M$ to $(x,\mathbf{a})$. 
For each $m$, we have that $(x^m,\mathbf{a}^m) \in V$, and therefore there exists a sequence of natural numbers $\{\ell(m,k)\}_{k=1}^{\infty}$ tending to infinity and a sequence of special finite labeled trajectories $(y^{m,k},b^{m,k}) \in \mathcal{S}_{\ell(m,k)}$ such that for each $m$ and $n$,
\[
\lim_{k} b^{m,k}_n = a^m_n, \quad \text{and} \quad \lim_{k} y^{m,k}_n = x^m_n.
\]
To complete the proof, we exhibit a sequence $\{\ell_j\}_{j=1}^{\infty}$ of natural numbers and a sequence $\{(z^j,c^j)\}_{j=1}^{\infty}$ of special finite labeled trajectories to demonstrate that $(x,\mathbf{a}) \in V$.

Let $j \geq 1$. First choose $m_j$ such that for all $n = 0,\dots,j$, we have $a^{m_j}_n = a_n$ and
\[
\bigl|x_n^{m_j} - x_n \bigr| < \frac{1}{2j}.
\]
Next choose $k_j$ (depending on $m_j$) such that $\ell(m_j,k_j) \geq j$ and for all $n = 0,\dots,j$, we have $b^{m_j,k_j}_n = a^{m_j}_n$ and
\[
\bigl|y^{m_j,k_j}_n - x^{m_j}_n \bigr| < \frac{1}{2j}.
\]
Finally, let $\ell_j = j$, and define 
\[
z^j = y^{m_j,k_j}_0,\dots,y_j^{m_j,k_j}
\]
and
\[
c^j = b^{m_j,k_j}_0 \dots b^{m_j,k_j}_j.
\]
%
Then $\{\ell_j\}_{j=1}^{\infty}$ tends to infinity and $\{(z^j,c^j)\}_{j=1}^{\infty}$ is a sequence of special finite labeled trajectories. Furthermore, for each $n$, we have that $\{(z^j_n,c^j_n)\}_{j=1}^{\infty}$ converges to $(x_n,a_n)$ in $[0,1] \times \mathcal{A}$.
We have thus exhibited the necessary sequences to establish that $(x,\mathbf{a}) \in V$.

To establish shift invariance, suppose that $(x,\mathbf{a}) \in V$. Let us show that $(\sigma(x),\sigma(\mathbf{a})) \in V$. There exists natural numbers $\{\ell_k\}_{k=1}^{\infty}$ and special finite trajectories $\{(y^k, a^k)\}_{k=1}^{\infty}$ that witness the fact that $(x,\mathbf{a}) \in V$. Define $z^k_n = y^k_{n+1}$ and $b^k_n = a^k_{n+1}$. Then the sequences $\{\ell_k-1\}_{k=1}^{\infty}$ and $\{(z^k,b^k)\}_{k=1}^{\infty}$ establish that $(\sigma(x),\sigma(\mathbf{a})) \in V$.
\end{proof}

As $V$ is invariant under the left shift, we define $\sigma_V : V \to V$ by letting $\sigma_V(x,\mathbf{a}) = (\sigma(x),\sigma(\mathbf{a}))$. 

In the proof of our main results, we use the joint space $V$ as an intermediary between the spaces $X$ and $\Sigma_M$. To make this connection precise, we define factor maps from $V$ onto each of $X$ and $\Sigma_M$.

\subsection{Factoring onto $\Sigma_M$} \label{Sect:FactoringOntoSigma}

Here we show that the joint space $V$ from Definition \ref{Defn:V} factors onto $\Sigma_M$.

\begin{defn} \label{Defn:Phi}
Let $F$ be in  $\mathcal{F}$ with associated trajectory space $X$, SFT $\Sigma_M$, and joint space $V$. Define the map $\phi : V \to \Sigma_M$ by the rule $\phi(x,\mathbf{a}) = \mathbf{a}$. 
\end{defn}

It is clear that $\phi$ is continuous and commutes with the left shift.


\begin{rmk} \label{Rmk:Rain}
Note that if $M(a,b) = 1$, then $f_a^{-1}(D_0(b)) \subset D_0(a)$. Furthermore, if $y \in D_0(b)$ and $x = f^{-1}_a(y)$, then $(x,y) \in G_0(a)$. Thus, if $a_0 \dots a_{\ell} \in \mathcal{L}_{\ell}$ and $y_{\ell} \in D_0(a_{\ell})$, then for each $n = 0, \dots, \ell-1$, we have
\[
y_n = f_{a_n}^{-1} \circ \dots \circ f_{a_{\ell-1}}^{-1}(y_{\ell}) \in D_0(a_n),
\]
and $(y_n,y_{n+1}) \in G_0(a_n)$. 
\end{rmk}

\begin{prop}
$\phi$ is surjective.
\end{prop}
\begin{proof}
Let $\mathbf{a} \in \Sigma_m$. 
For each $\ell \geq 1$, let $y^{\ell}_{\ell} \in D_0(a_{\ell})$ be arbitrary. For $n = 0, \dots, \ell-1$, let $y^{\ell}_n = f_{a_n}^{-1} \circ \dots \circ f_{a_{\ell-1}}^{-1}(y^{\ell}_{\ell})$. By Remark \ref{Rmk:Rain}, for each $n = 0, \dots, \ell-1$, we have $(y^{\ell}_n,y^{\ell}_{n+1}) \in G_0(a_n)$.  Also, for each $n$, we have that $y^{\ell}_{n} \in [0,1]$, which is sequentially compact. Thus, by a diagonal argument, there exists a subsequence $\{\ell_k\}_{k=1}^{\infty}$ tending to infinity such that for each $n$, there exists $x_n \in [0,1]$ such that
\[
\lim_{k \to \infty} y^{\ell_k}_n = x_n.
\]
Setting $x = (x_n)_{n=0}^{\infty}$, we see that $(x,\mathbf{a}) \in V$ and $\phi(x,\mathbf{a}) = \mathbf{a}$.
\end{proof}


The following proposition asserts that $\phi$ preserves the entropy of ergodic measures. Its proof is an adaptation of the proof of \cite[Theorem 4.1]{AlvinKelly2019}, and we provide it in Appendix \ref{Sect:PhiPreservesEntropy} for completeness.

\begin{prop} \label{Prop:PhiPreservesEntropy}
Let $\phi : V \to \Sigma_M$ be as in Definition \ref{Defn:Phi}. Furthermore, let $\mu \in \mathcal{M}_e(V,\sigma_V)$ and $\nu \in \mathcal{M}_e(\Sigma_M,\sigma_M)$ be such that $\nu = \mu \circ \phi^{-1}$. Then $h(\nu) = h(\mu)$. 
\end{prop}

\subsection{Factoring onto $X$} \label{Sect:FactoringOntoX}

Here we show that the joint space $V$ from Definition \ref{Defn:V} factors onto $X$.

\begin{defn} \label{Defn:Pi}
Let $F$ be in  $\mathcal{F}$ with associated trajectory space $X$, SFT $\Sigma_M$, and joint space $V$.
Define the map $\pi : V \to [0,1]^{\mathbb{Z}_+}$ by the rule $\pi(x,\mathbf{a}) = x$.
\end{defn}

It is clear that $\pi$ is continuous and commutes with the left shift. The following result shows that the image of $\pi$ is contained in $X$.

\begin{prop} \label{Lemma:PiWellDef}
Suppose that $(x,\mathbf{a}) \in V$. Then $x \in X$. 
\end{prop}
\begin{proof}
Let $n \geq 0$. Since $(x,\mathbf{a}) \in V$, there exists $y^k_n$ and $y^k_{n+1}$ such that $\lim_k y^k_n = x_n$, $\lim_k y^k_{n+1} = x_{n+1}$, and $(y^k_n,y^k_{n+1}) \in G(a_n)$. Since $G(a_n)$ is closed, we see that $(x_n,x_{n+1}) \in G(a_n)$ for each $n \geq 0$. Then by the definition of $X$, we have $x \in X$.
\end{proof}

By Proposition~\ref{Lemma:PiWellDef}, we have $\pi : V \to X$. Next we establish that $\pi$ in fact maps \textit{onto} $X$. First, 
%
let $V_0 \subset V$ be the set of points $(x,\mathbf{a}) \in V$ such that for each $n \geq 0$, we have $(x_n,x_{n+1}) \in G_0(a_n)$.

\begin{prop} \label{Lemma:Surjective}
For each $x \in X$, there exists a unique $\mathbf{a} \in \Sigma_M$ such that $(x,\mathbf{a}) \in V_0$. In particular, $\pi : V \to X$ is surjective.
\end{prop}
\begin{proof}
Let $x = (x_n)_{n=0}^{\infty} \in X$. 
Let $n \geq 0$. Since $(x_n, x_{n+1}) \in G(F)$ and $\{G_0(a) : a \in \mathcal{A}\}$ is a partition of $G(F)$, there is a unique element $a_n \in \mathcal{A}$ such that $(x_n,x_{n+1}) \in G_0(a_n)$. This uniquely defines a sequence $\mathbf{a} = (a_n)_{n=0}^{\infty}$. 

Let us show that $\mathbf{a} \in \Sigma_M$.
Since $\Sigma_M$ is a SFT defined by the matrix $M$, it suffices to show that for each $n \geq 1$, we have $M(a_{n-1},a_n) = 1$. 
Let $n \geq 1$. By construction, we have that $(x_{n-1},x_n) \in G_0(a_{n-1})$ and $(x_n,x_{n+1}) \in G_0(a_n)$. Then $x_n \in R_0(a_{n-1})$ and $x_n \in D_0(a_n)$, and therefore $D_0(a_n) \cap R_0(a_{n-1}) \neq \varnothing$. By the Markov property, we conclude that $D_0(a_n) \subset R_0(a_{n-1})$, and therefore $M(a_{n-1},a_n) = 1$, as desired.

Finally, note that $(x ,\mathbf{a}) \in V$. Indeed, for each $k \geq 1$, let $\ell_k = k$, $y^k_n = x_n$ and $b^k_n = a_n$. Then we have exhibited the necessary sequences to establish that $(x,\mathbf{a}) \in V$.
\end{proof}

\section{Constructing the bad sets}

Our aim is to show that under certain conditions, we can construct a Borel entropy conjugacy between $X$ and $\Sigma_M$. In this construction, we identify ``bad sets", on which the Borel entropy conjugacy map will not be defined. The main source of difficulty in constructing our Borel entropy conjugacy arises from the fact that points in the trajectory space that stay in the critical set $P$ can have multiple symbolic codings. In order to deal with this difficulty, we group such symbolic codings into the ``bad sets" and show that we have only removed sets of strictly smaller entropy that the full system. In fact, we carry out this process for each irreducible component of  $\Sigma_M$ separately. In the following section, we define the critical set of points in $X$ that cause us difficulty. Then in the following sections we analyze the irreducible components in detail and construct their bad sets. 

\subsection{The critical system}

Let $F$ be in $\mathcal{F}$ with trajectory space $X$, SFT $\sigma_M$, and joint system $V$. Consider the set of trajectories contained in the critical set $P$: 
\[
X_P = \bigl\{ x \in X : \forall n \geq 0, x_n \in P \bigr\}.
\]
Note that $X_P$ is closed and invariant under $\sigma_X$. 
We refer to $X_P$ as the critical system. 
Now let $Z = \pi^{-1}(X_P) \subset V$, and note that $Z$ is closed and invariant under $\sigma_V$. 
As we mentioned above, one of the main difficulties in relating $X$ and $\Sigma_M$ lies in the fact that $\pi$ may not be injective on $Z$ (or its pre-images under the shift).

\subsection{Irreducible components}

We find it useful to distinguish between the following types of irreducible components for Markov multi-maps.
\begin{defn}
Let $F$ be in $\mathcal{F}$ with associated SFT $\Sigma_M$.
Let $\mathcal{C} \subset \mathcal{A}$ be an irreducible component of the $M$-graph. We say that
\begin{itemize}
\item $\mathcal{C}$ is of Type I if $\mathcal{C} \subset \mathcal{A}_0$;
\item $\mathcal{C}$ is of Type II if $\mathcal{C} \subset \mathcal{A}_2$;
\item $\mathcal{C}$ is of Type III if it is not Type I or Type II.
\end{itemize}
\end{defn}
\begin{rmk}
Suppose $\mathcal{C}$ is of Type III. Then for each $i \in \{0,1,2\}$, we must have $\mathcal{C} \cap \mathcal{A}_i \neq 0$. In fact, there must exist allowable transitions in $\mathcal{C}$ from $\mathcal{A}_0$ to $\mathcal{A}_2$ (cross-over), from $\mathcal{A}_2$ to $\mathcal{A}_1$ (into a vertical line), and from $\mathcal{A}_1$ to $\mathcal{A}_0$ (out of vertical line). 
\end{rmk}

Let $\Sigma_M(\mathcal{C})$ denote the irreducible component of $\Sigma_M$ corresponding to $\mathcal{C}$.
Note that $\Sigma_M(\mathcal{C})$ is an SFT contained in $\Sigma_M$.
Also, for distinct irreducible components $\mathcal{C}_1$ and $\mathcal{C}_2$, we have that $\Sigma_M(\mathcal{C}_1)$ and $\Sigma_M(\mathcal{C}_2)$ are disjoint.
Let $V(\mathcal{C})$ denote $\phi^{-1}(\Sigma_M(\mathcal{C})) \subset V$, where $V$ is the joint system.
Also, let $Z(\mathcal{C}) = Z \cap V(\mathcal{C})$, where $Z = \pi^{-1}(X_P) \subset V$.

\subsection{Constructing the bad sets: Types I and III}

We now show the existence of our ``bad sets" off of which $\phi$ and $\pi$ are injective.
In the proof of the following proposition, we use the following immediate consequence of Lemma \ref{Lemma:McCartney}: if $(x,\mathbf{a}), (x,\mathbf{b}) \in V$ and $x_m \notin P$, then for each $n < m$, we have $a_n = b_n$. Also, for notation, for any word $w \in \mathcal{L}$ and any $\mathbf{a} \in \Sigma_M$, let 
\[
N_w(\mathbf{a}) = \bigl| \bigl\{ n \geq 0 : \sigma^{n}(\mathbf{a}) \in [w] \bigr\} \bigr|.
\]

\begin{prop} \label{Prop:TypeI}
Let $F$ be in $\mathcal{F}$ with SFT $\Sigma_M$ and joint system $V$.  Suppose that $\mathcal{C}$ is a Type I or Type III irreducible component of the $M$-graph such that $\htop(\Sigma_M(\mathcal{C}),\sigma_M|_{\Sigma_M(\mathcal{C})}) >0$. 
Then there exists words $u^c$ and $u^a$ in $\mathcal{L}(\mathcal{C})$ such that if 
\[
B_0 = \{ \mathbf{a} \in \Sigma_M(\mathcal{C}) : N_{u^c}(\mathbf{a}) < \infty \text{ or } N_{u^a}(\mathbf{a}) < \infty \},
\]
and $B = \phi^{-1}(B_0)$, 
then
\begin{enumerate}
\item $\Pre(Z(\mathcal{C}) )\subset B$,
\item  $\hprob(B) < \htop(V(\mathcal{C}), \sigma|_{V(\mathcal{C}})$, and
\item both $\pi$ and $\phi$ are injective on $V(\mathcal{C}) \setminus B$.
\end{enumerate}
\end{prop}
\begin{proof}
First, suppose that $\mathcal{C}$ is Type I. 
Since $F$ is in $\mathcal{F}$, it has a complete set of coding words, and we may select a coding word $u^c$ for $\mathcal{C}$. Similarly, since $F$ is in $\mathcal{F}$, it has a complete set of avoiding words, and then since $\mathcal{C}$ is of Type I, we may select an avoiding word $u^a$ for $\mathcal{C}$. 

Now suppose that $\mathcal{C}$ is Type III. 
Since $\mathcal{C}$ is of Type III, it contains a word $u^a = w_0 w_1$ such that $w_0 \in \mathcal{A}_0$ and $w_1 \in \mathcal{A}_1 \cup \mathcal{A}_2$. Note that $u^a$ is an avoiding word. Furthermore, since $\mathcal{C}$ is of Type III, it contains a symbol $u^c \in \mathcal{A}_1$. Note that $u^c$ is a coding word for $\mathcal{C}$.

For the remainder of the proof, we do not distinguish between whether $\mathcal{C}$ is Type I or Type III.

Then let
\[
B_0 = \{ \mathbf{a} \in \Sigma_M(\mathcal{C}) : N_{u^c}(\mathbf{a}) < \infty \text{ or } N_{u^a}(\mathbf{a}) < \infty \},
\]
and let $B = \phi^{-1}(B_0)$. (Note that $B_0$ and $B$ are invariant.)

To establish (1), let $(x,\mathbf{a}) \in \Pre(Z(\mathcal{C}))$. Then there exists $N$ such that for each $n\geq N$, we have $x \in I_{\sigma^{n}(\mathbf{a})} \cap P$. Since $u^a$ is an avoiding word, we see that $N_{u^a}(\mathbf{a}) < \infty$. It follows that $\mathbf{a} \in B_0$, and therefore $(x,\mathbf{a}) \in B$.

Now we establish (2).
Let $Y \subset \Sigma_M$ be the SFT obtained by forbidding $u^c$ and $u^a$. 
Since $\Sigma_M(\mathcal{C})$ is irreducible, it is entropy minimal. Then $\htop(Y,\sigma|_Y) < \htop(\Sigma_M(\mathcal{C}),\sigma_M|_{\Sigma_M(\mathcal{C})})$, as $Y$ is a strict subsystem of $\Sigma_M(\mathcal{C})$.
Suppose $\mu$ is an ergodic measure on $\Sigma_M(\mathcal{C})$ such that $\mu(B_0) =1$. As $\mu$ is ergodic and the words $u^c$ and $u^a$ appear only finitely for points in $B_0$, we must have $\mu([u^c]) = \mu([u^a]) = 0$, and therefore $\mu(Y) = 1$. Then by the variational principle, we see that $h(\mu) \leq \htop(Y,\sigma|_Y)$. Taking the supremum over all such $\mu$, we obtain that $\hprob(B_0) \leq \htop(Y,\sigma|_Y) < \htop(\Sigma_M(\mathcal{C}),\sigma_M|_{\Sigma_M(\mathcal{C})})$. 
Furthermore, since $\phi$ preserves the entropy of ergodic measures (by Proposition \ref{Prop:PhiPreservesEntropy}), we obtain that $\hprob(B) = \hprob(B_0) < \htop(\Sigma_M(\mathcal{C}),\sigma_M|_{\Sigma_M(\mathcal{C})}) = \htop(V(\mathcal{C}),\sigma_V|_{V(\mathcal{C})})$.

To show that $\phi$ is injective on $V(\mathcal{C}) \setminus B$, let $(x,\mathbf{a}) \in V(\mathcal{C}) \setminus B$, and suppose $(y,\mathbf{a}) \in V(\mathcal{C}) \setminus B$. Then $\mathbf{a} \notin B_0$, and in fact $\sigma^n(\mathbf{a}) \notin B_0$ for all $n \geq 0$. Let $n \geq 0$. Then $\sigma^n(\mathbf{a})$ contains the word $u^c$ infinitely many times, and therefore $I_{\sigma^n(\mathbf{a})}$ is a singleton (since $u^c$ is a coding word). Since we must have both $x_n \in I_{\sigma^n(\mathbf{a})}$ and $y_n \in I_{\sigma^n(\mathbf{a})}$, we conclude that $x_n = y_n$. As $n \geq 0$ was arbitrary, we have shown that $\phi$ is injective on $V(\mathcal{C}) \setminus B$.

To show that $\pi$ is injective on $V(\mathcal{C}) \setminus B$, let $(x,\mathbf{a}), (x,\mathbf{b}) \in V(\mathcal{C}) \setminus B$. Let $T = \{ m \geq 0 : \sigma^m(\mathbf{a}) \in [u^a] \}$. For each $m \in T$, we have that $x_m \in I_{\sigma^m(\mathbf{a})} \subset [0,1] \setminus P$ (since $u^a$ is an avoiding word). Since $(x,\mathbf{a}) \in V(\mathcal{C}) \setminus B$, the set $T$ must be infinite. Let $n \geq 0$. Since $T$ is infinite, there exists $m > n$ such that $m \in T$. Then $x_m \notin P$. By Lemma \ref{Lemma:McCartney}, we see that $a_n = b_n$. 
As $n \geq 0$ was arbitrary, we conclude that $\pi$ is injective on $V(\mathcal{C}) \setminus B$.
\end{proof}

\subsection{Constructing the bad sets: Type II}

We don't have to remove any bad sets from Type II components. Indeed, the following proposition establishes that $\pi$ and $\phi$ are injective on the union of all Type II components.
\begin{prop} \label{Prop:TypeII}
Let $F$ be in $\mathcal{F}$, and let 
\[
V_P = \bigcup_{\mathcal{C} \text{ of Type II}} V(\mathcal{C}).
\]
Then $\pi$ and $\phi$ are injective on $V_P$. 
\end{prop}
\begin{proof}
Suppose $(x,\mathbf{a}), (x,\mathbf{b}) \in V_P$. Then $a_n, b_n \in \mathcal{A}_2$ for all $n \geq 0$, and we must have $G(a_n) = \{(x_n,x_{n+1})\} = G(b_n)$ for all $n$. Therefore $a_n = b_n$ for all $n \geq 0$, and  $\pi$ is injective on $V_P$.

Suppose that $(x,\mathbf{a}), (y,\mathbf{a}) \in V_P$. Then $D(a_n)$ is a singleton for each $n$, and we must have $\{x_n\} = D(a_n) = \{y_n\}$ for all $n$. Therefore $x_n = y_n$ for all $n \geq 0$, and $\phi$ is injective on $V_P$.
\end{proof}

\section{Proof of the main result} \label{Sect:MainProof}

Now that we have constructed the bad sets for each type of irreducible component, we are ready to prove our main result on Borel entropy conjugacy.

\vspace{2mm}

\begin{PfofThmEntropyConjugacy}

Let $F$ be in  $\mathcal{F}$ with associated trajectory space $X$ and SFT $\Sigma_M$. Furthermore, let $V$ be the associated joint space, as in Definition \ref{Defn:V}, and let $\phi : V \to \Sigma_M$ and $\pi : V \to X$ be the maps defined in Definitions \ref{Defn:Phi} and \ref{Defn:Pi}, respectively. 

Since $F \in \mathcal{F}$, we have that $\Sigma_M$ has positive entropy.
Enumerate the irreducible components with positive entropy: $\mathcal{C}_1,\dots,\mathcal{C}_J$. For each $\mathcal{C}_j$ of Type I or Type III, let $B_j \subset V(\mathcal{C}_j)$ be the bad set given by Proposition \ref{Prop:TypeI}. For each $\mathcal{C}_j$ of Type II, let $B_j = \varnothing$. Furthermore, let 
\[
B_0 = \phi^{-1} \Biggl(  \Sigma_M \setminus \Biggl( \bigcup_{j = 1}^J \Sigma_M( \mathcal{C}_j )\Biggr) \Biggr).
\]
Then let
\[
B =  \bigcup_{j=0}^{J} B_j.
\]
For each $j = 1, \dots, J$, let $A_j = V(\mathcal{C}_j) \setminus B_j$. Note that
\begin{equation} \label{Eqn:Barron}
V \setminus B = \bigcup_{j=1}^J A_j.
\end{equation}

\begin{prop}\label{Proposition:Pi}
$\pi$ is injective on $V \setminus B$.
\end{prop}
\begin{proof}
Suppose that $(x,\mathbf{a}), (x,\mathbf{b}) \in V \setminus B$. By (\ref{Eqn:Barron}), there exists $i,j$ such that $(x,\mathbf{a}) \in A_i$ and $(x,\mathbf{b}) \in A_j$. If $x_n \notin P$ for infinitely many $n$, then $\mathbf{a} = \mathbf{b}$ by Lemma \ref{Lemma:McCartney}. 

Now suppose that $x_n \in P$ for all by finitely many $n$. Then there exists $N$ such that $\sigma^{N}(x) \in X_P$. Hence $(x,\mathbf{a}), (x,\mathbf{b}) \in \Pre(Z)$, and therefore $\mathcal{C}_i$ and $\mathcal{C}_j$ must be of Type II (since $A_k \cap Z = \varnothing$ whenever $\mathcal{C}_k$ is of Type I or Type III). Since $\pi$ is injective on $V_P$,
we conclude that $\mathbf{a} = \mathbf{b}$.
\end{proof}

\begin{prop}\label{Proposition:Phi}
$\phi$ is injective on $V \setminus B$.
\end{prop}
\begin{proof}
Suppose that $(x,\mathbf{a}), (y,\mathbf{a}) \in V \setminus B$. By (\ref{Eqn:Barron}), there exists $i,j$ such that $(x,\mathbf{a}) \in A_i$ and $(y,\mathbf{a}) \in A_j$. Then $\mathbf{a} \in \Sigma_M(\mathcal{C}_i) \cap \Sigma_M(\mathcal{C}_j)$. Since distinct irreducible components are disjoint, we see that $i = j$. Since $\phi$ is injective on $A_i$, we conclude that $x = y$.
\end{proof}

\begin{lemma}\label{Lemma:h(X)=h(V)}
$\htop(X,\sigma_X) = \htop(V,\sigma_V)$.
\end{lemma}
\begin{proof}
First, by Proposition \ref{Lemma:PiWellDef}, we have that $\pi$ is a factor map. Since entropy cannot increase under a factor map, $\htop(X,\sigma_X) \leq \htop(V,\sigma_V)$.

Now fix $\mathcal{C}_i$ such that $\htop(V,\sigma_V) = \htop(V(\mathcal{C}_i),\sigma|_{V(\mathcal{C}_i)})$, and let $\mu$ be an ergodic measure on $V(\mathcal{C}_i)$ such that $h(\mu) = \htop(V(\mathcal{C}_i),\sigma|_{V(\mathcal{C}_i)})$ (which exists $\Sigma_M(\mathcal{C}_i)$ has a measure of maximal entropy and $\phi$ preserves entropy). Since $\hprob(B_i) < \htop(V(\mathcal{C}_i),\sigma|_{V(\mathcal{C}_i)}) = h(\mu)$, we must have that $\mu(B_i) = 0$. Then $\pi$ is injective a set of full $\mu$-measure, and therefore $\pi$ is an isomorphism from $\mu$ to $\pi \mu = \mu \circ \pi^{-1}$. In particular, $\htop(V,\sigma|_V) = \htop(V(\mathcal{C}_i),\sigma|_{V(\mathcal{C}_i)}) = h(\mu) = h(\pi \mu) \leq \htop(X,\sigma|_X)$, where the last inequality follows from the Variational Principle. We have now shown that $\htop(X,\sigma_X) = \htop(V,\sigma_V)$.
\end{proof}

Let $A_X = \pi( V \setminus B)$ and $A_{\Sigma} = \phi(V \setminus B)$. 
\begin{prop}
$A_X$ is Borel, $\sigma(A_X) = A_X$, and $\hprob(X \setminus A_X) < \htop(X,\sigma|_X)$.
\end{prop}
\begin{proof}
Consider $A_i$. First suppose that $\mathcal{C}_i$ is of Type II. Then $A_i = V(\mathcal{C}_i)$, which is compact. Thus $\pi(A_i)$ is also compact. In particular, $\pi(A_i)$ is closed and hence Borel.

Now suppose that $\mathcal{C}_i$ is of Type I or Type III. Then there exist words $u$ and $v$ in $\mathcal{L}(\mathcal{C}_i)$ (in particular, a coding word and any avoiding word) such that
\[
A_i = \bigcap_{N} \Biggl[ \Biggl( \bigcup_{n \geq N} \sigma^{-n}[u] \Biggr) \cap \Biggl( \bigcup_{n \geq N} \sigma^{-n}[v] \Biggr) \Biggr].
\]
Note that for each $n$, the sets $[u]$ and $[v]$ are compact. Hence $\pi[u]$ and $\pi[v]$ are compact and in particular closed. Then
\[
\pi(A_i) = \bigcap_{N} \Biggl[ \Biggl( \bigcup_{n \geq N} \sigma^{-n}\pi[u] \Biggr) \cap \Biggl( \bigcup_{n \geq N} \sigma^{-n}\pi[v] \Biggr) \Biggr],
\]
which shows that $\pi(A_i)$ is Borel. Finally, since $A_X = \cup_i \pi(A_i)$, we conclude that $A_X$ is Borel.

For each $A_i$, we have $\sigma(A_i) = A_i$, and therefore $\sigma( \pi(A_i) ) = \pi ( \sigma(A_i) ) = \pi(A_i)$. As $A_X = \sqcup_i \pi(A_i)$, we see that $\sigma(A_X) = A_X$.

Let $\nu$ be an ergodic invariant measure on $X$ such that $\nu(X \setminus A_X) >0$. Since $A_X$ is invariant and $\nu$ is ergodic, we have that $\nu(X \setminus A_X) = 1$, and therefore $\nu(A_X) = 0$.  Also, $\nu$ is supported on some set of the form $\pi(B_i)$, with $1 \leq i \leq J$. Let $\mu$ be an ergodic measure on $V$ such that $\pi \mu = \nu$. Then $\mu(A) \leq \mu( \pi^{-1} \pi(A) ) = \nu(A_X) = 0$, and $\mu(V(\mathcal{C}_i)) = 1$. Therefore $\mu(B_i) = 1$. Finally, we observe that $h(\nu) \leq h(\mu) \leq \hprob(B_i) \leq \max_i \hprob(B_i)$. Since the right hand side is strictly less than $\htop(V,\sigma_V)$, which equals $\htop(X,\sigma_X)$ by Lemma~\ref{Lemma:h(X)=h(V)},  we conclude that $\hprob(X \setminus A_X) < \htop(X,\sigma|_X)$.
\end{proof}

The following proposition may be quite easily deduced from the definitions, and we omit its proof.
\begin{prop}
$A_{\Sigma}$ is Borel, $\sigma(A_{\Sigma}) = A_{\Sigma}$, and $\hprob(\Sigma_M \setminus A_{\Sigma}) < \hprob(\Sigma_M)$. 
\end{prop}

Now define $\psi  = \pi|_{V \setminus B} \circ \phi|_{V \setminus B}^{-1} : A_{\Sigma} \to A_X$, which will serve as our Borel entropy conjugacy map.
\begin{prop}
$\psi$ is bijective, bi-measurable, and commutes with the left shift.
\end{prop}
\begin{proof}
Taken together, Propositions \ref{Proposition:Pi} and \ref{Proposition:Phi} yield that $\psi$ is bijective. Let $E \subset A_X$ be Borel. Then $\pi^{-1}(E) \cap (V \setminus B)$ is Borel. Also, since $\phi|_{V \setminus B}$ is an injective continuous map on the Borel set $V \setminus B$, it maps Borel sets to Borel sets. Therefore  $\psi^{-1}(E) = \phi|_{V \setminus B}( \pi^{-1}(E) \cap (V \setminus B))$ is Borel measurable. Therefore $\psi$ is Borel measurable. An analogous argument shows that $\phi^{-1}$ is also Borel measurable. Finally, since $\pi$ and $\phi$ commute with the left shift, $\psi$ also commutes with the left shift.
\end{proof}

By the previous propositions, we conclude that $\psi$ is the desired Borel entropy conjugacy between $X$ and $\Sigma_M$.
\end{PfofThmEntropyConjugacy}

\section{Sufficient conditions for $F$ to be in $\mathcal{F}$} \label{Sect:Sufficient}

Now that we have proved Theorem~\ref{Thm:EntropyConj}, we wish to highlight its utility by establishing some straightforward conditions that are sufficient for a Markov multi-map $F$ to be in the family $\mathcal{F}$. We focus on the case where there is an irreducible component $\C$ that contains all of $\mathcal{A}_0$. For single-valued functions, this condition amounts to the topological transitivity of the system.

\begin{defn}
	Suppose $\C$ is an irreducible component. We say that \emph{$F$ codes for points on $\C$} if
	\[
	\limsup_n\left\{\ell\left(I_{a_0\cdots a_n}\right)\colon\mathbf{a}\in\Sigma_M(\C)\right\}=0.
	\]
\end{defn}

\begin{lemma}\label{Lemma:CodesforPoints}
	Suppose $\C$ is an irreducible component with $\mathcal{A}_0\subset\C$. If $F$ codes for points on $\C$, then $\C$ has a coding word and an avoiding word.
\end{lemma}
\begin{proof}
	Since $F$ codes for points on $\C$, by definition, we must have that $I_\mathbf{a}$ is a singleton for all $\mathbf{a}\in\Sigma_M(\C)$. Therefore every word in $\mathcal{L}(\C)$ is a coding word. To see that $\C$ also has an avoiding word, we consider two cases.
	
	Case 1: Suppose $\mathcal{A}_0$ is a strict subset of $\C$. Then $\C$ must be a Type III component, and we showed in the proof of Proposition~\ref{Prop:TypeI} that every Type III component has an avoiding word.
	
	Case 2: Suppose $\mathcal{A}_0=\C$. We have shown that $\C$ has a coding word, so there must exist $a,b\in\mathcal{A}_0$ such that $ab\in\mathcal{L}(\mathcal{A}_0)$ and $I_{ab}$ is a strict subset of $I_a=D(a)$. This would imply that $D(b)$ is a strict subset of $R(a)$.
	
	By the definition of a Markov multi-map, $D(b)$ is an interval between adjacent elements of the partition $P$. It follows that $P$ partitions $R(a)$ into at least two intervals, so there exist distinct elements $p_i,p_j\in P$ such that $[p_i,p_{i+1}]\cup[p_j,p_{j+1}]\subseteq R(a)$. Then there must be $b_1,b_2\in\mathcal{A}_0$ such that $D(b_1)=[p_i,p_{i+1}]$ and $D(b_2)=[p_j,p_{j+1}]$.
	
	The interval $I_{ab_1}$ is a strict subset of $I_a$, so it contains at most one endpoint of $D(a)$. Since $\mathcal{A}_0$ is irreducible, there exists $u\in\mathcal{L}(\mathcal{A}_0)$ such that $ab_1ua\in\mathcal{L}(\mathcal{A}_0)$. The interval $I_{ab_1ua}$ is contained in $I_{ab_1}$, so it contains at most one endpoint of $I_a$. Then $I_{ab_1uab_1}$ and $I_{ab_1uab_2}$ are non-overlapping, so at least one of them is disjoint from $P$. Therefore $\mathcal{A}_0$ has an avoiding word.	
\end{proof}

Next we define what it means for $F$ to be uniformly expanding on $\C$, and we show that if that is the case, then $F$ codes for points on $\C$. Recall that for each $a\in\mathcal{A}$, we have a well-defined function $f_a^{-1}\colon R(a)\to D(a)$, and if $u=a_0\cdots a_n\in\mathcal{L}$, then we define $f_u^{-1}=f_{a_0}^{-1}\circ\cdots\circ f_{a_n}^{-1}$.

\begin{defn}
	Suppose $\C$ is an irreducible component. We say that $F$ is uniformly expanding on $\C$ if there exists $N\in\N$ such that
	\[
	\sup\left\{\left|(f_u^{-1})'(x)\right|\colon u\in\mathcal{L}_N(\C),x\in D(a_0)\right\}<1.
	\]
\end{defn}

\begin{lemma}\label{Lemma:UniformlyExpanding}
	Let $\C$ be an irreducible component.	If $F$ is uniformly expanding on $\C$, then $F$ codes for points on $\C$.
\end{lemma}

\begin{proof}
	Let $0<\lambda<1$ such that $|(f_u^{-1})'(x)|<\lambda$ for all $u=a_0\cdots a_N\in\mathcal{L}_N(\C)$ and $x\in R(a_N)$. Then $\ell(I_u)<\lambda\ell(R(a_N))\leq\lambda$. It follows that for all $k\geq 1$ and $u\in\mathcal{L}_{kN}(\C)$, we have $\ell(I_u)<\lambda^k$. Therefore $F$ codes for points on $\C$.
\end{proof}

By combining these results, we arrive at the following sufficient condition for $F$ to be in $\mathcal{F}$.

\begin{cor}\label{Cor:UE}
Let $F$ be a properly parametrized Markov multi-map with associated SFT $\Sigma_M$. Suppose that $\htop(\Sigma_M,\sigma_M)>0$, and furthermore there is an irreducible component $\C$ with $\mathcal{A}_0\subset\C$. If $F$ is uniformly expanding on $\C$, then $F\in\mathcal{F}$, and hence $(X,\sigma_X)$ is entropy conjugate to $(\Sigma_M,\sigma_M)$.
\end{cor}
\begin{proof}
	By Lemma~\ref{Lemma:UniformlyExpanding} and Lemma~\ref{Lemma:CodesforPoints}, the component $\C$ has a coding word and an avoiding word. Since $\mathcal{A}_0\subset\C$, no irreducible component (except possibly $\C$) could be contained in $\mathcal{A}_0$, so $F$ has a complete set of coding words and a complete set of avoiding words. Thus $F\in\mathcal{F}$, so by Theorem~\ref{Thm:EntropyConj}, $(X,\sigma_X)$ is entropy conjugate to $(\Sigma_M,\sigma_M)$.
\end{proof}

\section{Realization of entropies} \label{Sect:Realization}

We now prove Theorem~\ref{Thm:Realization}, which we restate here.

\begin{realizationthm}
The set $\mathcal{H}(\mathcal{F})$ is equal to the set of all positive rational multiples of logarithms of Perron numbers.
\end{realizationthm}
\begin{proof}
	It suffices to show that for any irreducible SFT with positive entropy, there is a Markov multi-map in $\mathcal{F}$ with the same entropy.
	
	
	Let $\Sigma_M$ be an irreducible SFT with positive entropy associated with the $n\times n$ matrix $M$. Since $\Sigma_M$ has positive entropy, there is one row of $M$ with (at least) two ones. After possibly permuting the alphabet, suppose the first row has a one in columns $k$ and $k+1$.
	
	Now we define a Markov multi-map $F$ on the interval $[1,n+2]$ in terms of its graph. (To illustrate our construction, we give a specific matrix $M$ in Example~\ref{Example:Realization}, and we show that graph of the corresponding multi-map in Figure~\ref{Figure:Realization}.) 
	
	Let 
	\[
	P=\left\{1,1+\frac{1}{2},2,2+\frac{1}{2},\ldots,n,n+\frac{1}{2},n+1,n+\frac{3}{2},n+2\right\}.
	\]
	For each $i,j\in\{1,\ldots,n\}$, we associate the rectangle $[i,i+1/2]\times[j,j+1/2]$ with the matrix entry $M(i,j)$. If $M(i,j)=1$, we include (in the graph of $F$) a straight line connecting the bottom left corner $(i,j)$ to the top right corner $(i+1/2,j+1/2)$, and if $M(i,j)=0$, we leave that rectangle empty. There is one exception to this rule however. We have assumed that $M(1,k)=M(1,k+1)=1$, and instead of including two separate graphs in those two rectangles, we include one line connecting $(1,k)$ to $(1+1/2,k+3/2)$. 
	
	In this way we would have the graph of a multi-map with domain $\bigcup_{i=1}^n[i,i+1/2]$. We need $F$ to be defined on all of $[1,n+2]$, so we next define the graph in each rectangle of the form $[i+1/2,i+1]\times[n+3/2,n+2]$, where $i\in\{1,\ldots,n\}$. In these rectangles, we include a straight line connecting the points $(i+1/2,n+3/2)$ and $(i+1,n+2)$. Then finally, in each of the rectangles $[n+1,n+3/2]\times[n+3/2,n+2]$ and $[n+3/2,n+2]\times[n+3/2,n+2]$, we include a straight line connecting the bottom left corner to the top right corner.
	
	We have described the graph of $F$, but in order to show it is a Markov multi-map in $\mathcal{F}$ we should specify the indexing set $\mathcal{A}$ and identify a coding word and an avoiding word. Let $\mathcal{C}_0$ be a labeling of all of the straight lines that correspond to ones in the matrix $M$. (Recall that the cardinality of $\mathcal{C}_0$ will be one less than the number of ones in $M$, because the ones in the $(1,k)$ and $(1,k+1)$ entries correspond to just one straight line in the graph.) Then let $\mathcal{B}_0$ be the additional straight lines whose ranges were all $[n+3/2,n+2]$, and define $\mathcal{A}_0=\mathcal{C}_0\cup\mathcal{B}_0$.
	
	Each of the straight lines we considered have a bottom left endpoint and a top right endpoint. Let $\mathcal{C}_2$ and $\mathcal{B}_2$ be the collections of these left and right endpoints, respectively, and let $\mathcal{A}_2=\mathcal{C}_2\cup\mathcal{B}_2$. Finally let $\mathcal{A}_1=\varnothing$.
	
	Then $\mathcal{C}_0$ is a Type I irreducible component whose corresponding SFT has the same entropy as $\Sigma_M$. To complete the proof, we  show that if $\Sigma(\mathcal{A})$ and $\Sigma(\C_0)$ are the SFTs associated with $\mathcal{A}$ and $\C_0$ respectively, then $\htop(\Sigma(\mathcal{A}))=\htop(\Sigma(\C_0))$. Towards this end, we show that the symbols in $\mathcal{B}_0,\mathcal{C}_2$, and $\mathcal{B}_2$ do not increase the entropy.
	
	Let $b_0\in\mathcal{B}_0$ represent the straight line in $[n+3/2,n+2]\times[n+3/2,n+2]$, then any $b\in\mathcal{B}_0$ can only be followed by $b_0$. This means $\hprob(\Sigma(\mathcal{A}_0))=\hprob(\Sigma(\mathcal{C}_0))$. Now we consider $\mathcal{C}_2$ and $\mathcal{B}_2$. Each of these individually follows nearly the same pattern as $\mathcal{A}_0$ with only one difference. Let $a^*\in\mathcal{A}_0$ correspond to the straight line in $[1,1/2]\times[k,k+3/2]$, and let $c^*\in\mathcal{C}_2$ and $b^*\in\mathcal{B}_2$ correspond to the respective endpoints of this line. The symbol $a^*$ can be followed by any symbol whose domain is $[k,k+1/2]$ or $[k+1,k+3/2]$. On the other hand $c^*$ can only be followed by points whose first coordinate is $k$, and $b^*$ can only be followed by points whose first coordinate is $k+3/2$. The SFTs corresponding to $\mathcal{C}_2$ and $\mathcal{B}_2$ are disjoint from one another and are invariant. It follows that the entropy contributed by these sets is less than or equal to the entropy from $\mathcal{C}_0$.
	
	All of this shows that $F$ is a Markov multi-map whose associated SFT has the same entropy as $(\Sigma_M,\sigma_M)$. It only remains to show that $F\in\mathcal{F}$. The only irreducible component in $\mathcal{A}_0$ with positive entropy is $\mathcal{C}_0$. We must show it has a coding word and an avoiding word. Once again let $a^*\in\mathcal{A}_0$ correspond to the straight line in $[1,1/2]\times[k,k+3/2]$. The range $R(a^*)$ is partitioned by $P$ into three non-overlapping intervals, so every occurrence of $a^*$ in a word $u$ decreases the length of $I_u$ by a factor of 3. It follows that $a^*$ is a coding word.
	
	We can also use $a^*$ to construct an avoiding word. Let $u=u_1\cdots u_m\in\mathcal{L}(\mathcal{C}_0)$ be any word such that $a^*ua^*\in\mathcal{L}(\mathcal{C}_0)$. The interval $I_{a^*u}$ is a strict subset of $[1,1+1/2]$, so it contains at most one point of $P$. There is then an element $b\in\mathcal{C}_0$ such that $I_{a^*ua^*b}$ is disjoint from $P$ and hence an avoiding word. Therefore $F\in\mathcal{F}$.
\end{proof}

\begin{example}\label{Example:Realization}
	Consider the $3\times 3$ matrix
	\[
	M=
	\left(
	\begin{array}{ccc}
	0 & 1 & 1\\
	1 & 0 & 0\\
	1 & 1 & 0
	\end{array}
	\right)_.
	\]
	Using the method outlined in the proof of Theorem~\ref{Thm:Realization}, this matrix would yield the graph pictured in Figure~\ref{Figure:Realization}. Recall  we stipulated that there must be at least two adjacent 1s in the first row of the matrix. These appear in the second and third rows of $M$. This gives us a line in the graph connecting the points $(1,2)$ and $(1.5,3.5)$.
	
	For the rest of the graph, note that if we rotate the matrix counter-clockwise ninety degrees, then the pattern of 1s in the matrix matches the pattern of lines in the lower portion of the graph.
	
	\begin{center}
		\begin{figure}
			\begin{tikzpicture}[scale=1.5]
			\draw[dotted] (1,1)node[below]{\small1} -- (5,1) node[below]{\small5} -- (5,5) -- (1,5)node[left]{\small5} -- (1,1)node[left]{\small1};
			\foreach \x in {1,1.5,2,2.5,3,3.5,4,4.5} \draw[dotted] (\x,5) -- (\x,1)node[below]{\small$\x$};
			\foreach \x in {1,1.5,2,2.5,3,3.5,4,4.5} \draw[dotted] (5,\x) -- (1,\x)node[left]{\small$\x$};
			\draw[very thick] 
			(1,2) -- (1.5,3.5)
			(2,1) -- (2.5,1.5)
			(3,1) -- (3.5,1.5)
			(3,2) -- (3.5,2.5)
			(1.5,4.5) -- (2,5)
			(2.5,4.5) -- (3,5)
			(3.5,4.5) -- (4,5)
			(4,4.5) -- (4.5,5)
			(4.5,4.5) -- (5,5)
			;
			\end{tikzpicture}
			\caption{Markov multi-map from Example~\ref{Example:Realization}}\label{Figure:Realization}
		\end{figure}
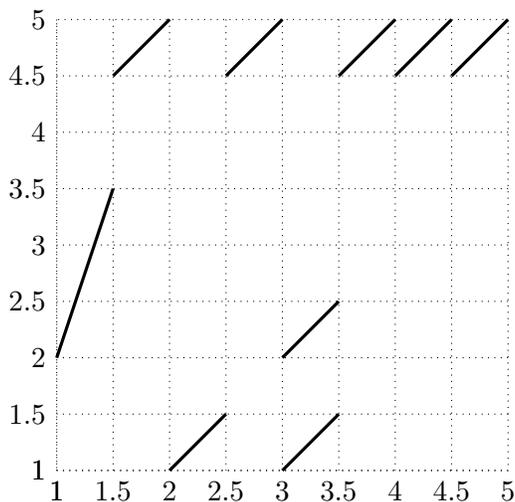
	\end{center}
\end{example}

\section{Examples}\label{Section:Examples}

We show various examples demonstrating the utility of our results. We begin by showing that Theorem~\ref{Thm:EntropyConj} generalizes the well-known result for the case that $F$ is single-valued.

\begin{example}\label{Example:UEMarkovMap}
Suppose $F$ is any uniformly expanding (single-valued) Markov map. Then Corollary~\ref{Cor:UE} recovers the well-known fact that $F$ is entropy conjugate to its combinatorial SFT.
\end{example}

Next we give an example of a Markov multi-map that is not uniformly expanding but still satisfies the hypotheses of Theorem~\ref{Thm:EntropyConj}.

\begin{example}\label{Example:NotUE}
Let $P=\{0,1/3,2/3,1\}$, $\mathcal{A}_0=\{1,2,3,4\}$, $\mathcal{A}_1=\emptyset$, and $\mathcal{A}_2=\{5,\ldots,10\}$. Let

\begin{minipage}{0.5\textwidth}
\begin{align*}
D_1&=[0,1/3]\\
D_2&=[1/3,2/3]\\
D_3&=[1/3,2/3]\\
D_4&=[2/3,1]\\
\end{align*}
\end{minipage}
\begin{minipage}{.5\textwidth}
\begin{align*}
R_1&=[1/3,2/3]\\
R_2&=[0,2/3]\\
R_3&=[2/3,1]\\
R_4&=[1/3,2/3].\\
\end{align*}
\end{minipage}

For each $a\in\{1,2,3\}$, let $G(a)$ be a straight line from the bottom left corner to the top right corner of $D(a)\times R(a)$, and let $G(4)$ be a straight line from the top left to the bottom right of $D(4)\times R(4)$. Then we define $G(5)=\{(0,1/3)\}$, $G(6)=\{(1/3,2/3)\}$, $G(7)=\{(2/3,1)\}$, $G(8)=\{(1/3,0)\}$, $G(9)=\{(2/3,1/3)\}$, $G(10)=\{(1,1/3)\}$ (all of the endpoints of $G(1),\ldots,G(4)$). This defines a Markov multi-map whose graph is pictured in Figure~\ref{Figures}.

Then $\mathcal{A}_0$ and $\mathcal{A}_2$ are both irreducible components with $\mathcal{A}_0$ (a Type I component) having greater entropy. The graph $G(3)$ has slope 2, but the rest of the graphs $G(1),G(2)$, and $G(4)$ have slope 1, so $F$ is not uniformly expanding on $\mathcal{A}_0$. However, $3\in\mathcal{L}_0$ is a coding word, because each time the symbol 3 appears in a word $u\in\mathcal{L}$, the length of the interval $I_u$ is divided in half. Also $331\in\mathcal{L}_2$ is an avoiding word, because $I_{331}=[3/6,7/12]$. Thus $F\in\mathcal{F}$, so by Theorem~\ref{Thm:EntropyConj} $(X,\sigma_X)$ is entropy conjugate to $(\Sigma_M,\sigma_M)$.
\end{example}

\begin{center}
	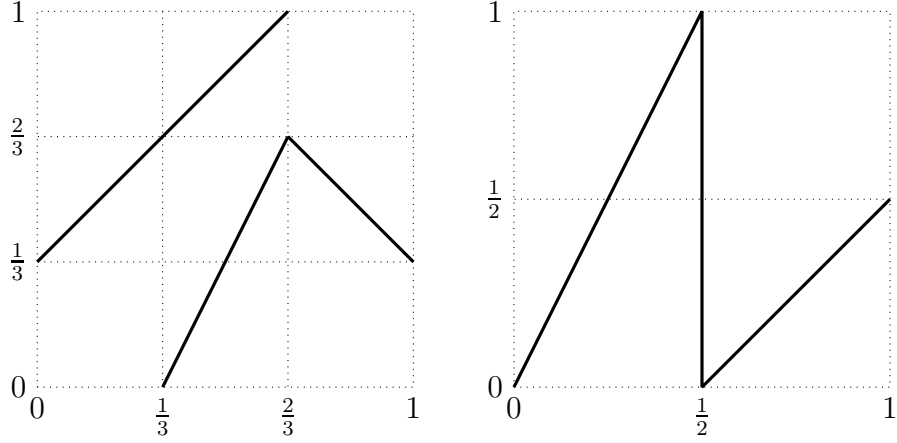
\begin{figure}
		\begin{minipage}{0.49\textwidth}
			\begin{tikzpicture}[scale=5]
			\draw[dotted] (0,0)node[below]{0} -- (1,0)node[below]{1} -- (1,1) -- (0,1)node[left]{1} -- (0,0)node[left]{0};
			\foreach \x in {1,2} \draw[dotted] ({\x/3},1) -- ({\x/3},0)node[below]{$\frac{\x}{3}$};
			\foreach \x in {1,2} \draw[dotted] (1,{\x/3}) -- (0,{\x/3})node[left]{$\frac{\x}{3}$};
			\draw[very thick] 
			(0,1/3) -- (1/3,2/3) 
			(1/3,0) -- (2/3,2/3)
			(1/3,2/3) -- (2/3,1)
			(2/3,2/3) -- (1,1/3);
			\end{tikzpicture}
		\end{minipage}
		\begin{minipage}{0.49\textwidth}
			\begin{tikzpicture}[scale=5]
			\draw[dotted] (0,0)node[below]{0} -- (1,0)node[below]{1} -- (1,1) -- (0,1)node[left]{1} -- (0,0)node[left]{0};
			\draw[dotted] (1/2,1) -- (1/2,0)node[below]{$\frac{1}{2}$};
			\draw[dotted] (1,1/2) -- (0,1/2)node[left]{$\frac{1}{2}$};
			\draw[very thick, join=bevel] (0,0) -- (1/2,1) -- (1/2,0) -- (1,1/2);
			\end{tikzpicture}
		\end{minipage}
		\caption{Markov multi-maps from Example~\ref{Example:NotUE} (left) and Example~\ref{Example:TypeIII} (right)}\label{Figures}
	\end{figure}
\end{center}

Next we show an example with a Type III irreducible component.

\begin{example}\label{Example:TypeIII}
	Let $P=\{0,1/2,1\}$, $\mathcal{A}_0=\{1,2\}$, $\mathcal{A}_1=\{3,4\}$, and $\mathcal{A}_2=\{5,6,7,8,9\}$. Let
	
	\begin{minipage}{0.5\textwidth}
		\begin{align*}
		D_1&=[0,1/2]\\
		D_2&=[1/2,1]\\
		\end{align*}
	\end{minipage}
	\begin{minipage}{.5\textwidth}
		\begin{align*}
		R_1&=[0,1]\\
		R_2&=[0,1/2]\\
		\end{align*}
	\end{minipage}
	
	Let $G(1)$ be the straight line connecting $(0,0)$ and $(1/2,1)$, and let $G(2)$ be the straight line connecting $(1/2,0)$ and $(1,1/2)$. Then let $G(3)$ and $G(4)$ be the vertical lines $\{1/2\}\times[0,1/2]$ and $\{1/2\}\times[1/2,1]$ respectively. Finally, define $G(5),\ldots,G(9)$ so that they are the endpoints of the graphs of $G(1),\ldots,G(4)$. The graph of this Markov multi-map is pictured in Figure~\ref{Figures}.
	
	In this case, two symbols from $\mathcal{A}_2$ represent the points $\{(0,0)\}$ and $\{(1/2,0)\}$. For simplicity, say these are $G(8)$ and $G(9)$. Then $\C=\{1,\ldots,7\}$ is a Type III irreducible component which means it must have a coding and an avoiding word. In this case, we can use $u=13\in\mathcal{L}_1$ as both a coding and an avoiding word, because $I_{13}=\{1/4\}$. Since there is no Type I component, we automatically have $F\in\mathcal{F}$.
\end{example}


Finally we give an example that does not satisfy our hypotheses, and for which $\htop(\Sigma_M,\sigma_M)$ is strictly greater than $\htop(X,\sigma_X)$.

\begin{example}
Define a Markov multi-map as follows.
Let $P=\{0,1\}$, $\mathcal{A}_0=\{1,2\}$, $\mathcal{A}_1=\varnothing$, and $\mathcal{A}_2=\{3,4\}$. Let $D(1)=D(2)=R(1)=R(2)=[0,1]$. Let $f_1,f_2\colon[0,1]\to[0,1]$ be defined by $f_1(x)=x^2$ and $f_2(x)=x^3$. Let $G(3)=\{(0,0)\}$, and $G(4)=\{(1,1)\}$. 

Then the only non-trivial irreducible component is $\mathcal{A}_0=\{1,2\}$, and $\Sigma_M(\mathcal{A}_0)$ is the full shift on two symbols, which has entropy $\log 2$. However, the only non-wandering points of $(X,\sigma_X)$ are the fixed points $(0,0,\ldots)$ and $(1,1,\ldots)$, so $\htop(X,\sigma_X)=0$.

Note that $I_u=[0,1]$ for all $u\in\mathcal{A}_0$, so this multi-map has neither coding words nor avoiding words. Thus $F\notin\mathcal{F}$, and Theorem~\ref{Thm:EntropyConj} does not apply.
\end{example}
%
%
%
%
%

\appendix

\section{Proof of Proposition 5.6} \label{Sect:PhiPreservesEntropy} 

Here we aim to prove Proposition \ref{Prop:PhiPreservesEntropy}. First, we recall the result of Katok \cite{Katok1980} relating the measure-theoretic entropy of an ergodic measure to Bowen balls. Consider a compact metric space $(\mathcal{X},d)$ and a continuous transformation $T : \mathcal{X} \to \mathcal{X}$. For $n \geq 1$, define the metric $d_n$ on $\mathcal{X}$ by setting
\[
d_n(x,y) = \max \bigl\{ d \bigl( T^k(x), T^k(y) \bigr) : k = 0,\dots, n-1 \bigr\}.
\]
For $\epsilon >0$, an $(n,\epsilon)$-ball is a ball of radius $\epsilon$ with respect to the metric $d_n$. 
Now let $\mu$ be in $\mathcal{M}_e(\mathcal{X},T)$. For $\alpha \in (0,1)$, $\epsilon >0$,  and $n \geq 1$, let $s(T,n,\epsilon,\alpha)$ denote the minimal cardinality of a collection of $(n,\epsilon)$-balls whose union has $\mu$-measure at least $\alpha$. Katok showed that 
\[
h(\mu) = \lim_{\epsilon \to 0^+} \limsup_n \frac{1}{n} \log s(T,n,\epsilon,\alpha).
\]

Let us now prove that the factor map $\phi : V \to \Sigma_M$ preserves the entropy of all ergodic measures.

\begin{PfofPropPhiPreservesEntropy}

As entropy cannot increase under factor maps, we have $h(\nu) \leq h(\mu)$. To complete the proof, we establish the reverse inequality.
Fix $\alpha \in (0,1)$.  For $n \geq 1$, let $r(n,\alpha)$ denote the minimal cardinality of a set of words $W \subset \mathcal{L}_n$ such that
\begin{equation} \label{Eqn:Child}
\nu\biggl( \bigcup_{w \in W} [w] \biggr) \geq \alpha.
\end{equation}
Let $\epsilon >0$. For $n \geq 1$, let $s(n,\epsilon,\alpha)$ denote the minimal cardinality of a collection $\mathcal{U}$ of $(n,\epsilon)$ balls in $V$ such that
\[
\mu \biggl( \bigcup_{U \in \mathcal{U}} U \biggr) \geq \alpha.
\]

Now let $n \geq 1$. Select a set $W = \{w_1,\dots,w_K\} \subset \mathcal{L}_n$ with cardinality $K = r(n,\alpha)$ and satisfying (\ref{Eqn:Child}). By the construction given in the proof of \cite[Theorem 4.1]{AlvinKelly2019}, for each $k$, there exists a collection $\mathcal{U}_k$ of $(n,\epsilon)$ balls in $V$ such that 
\[
\phi^{-1}([w_k]) \subset \bigcup_{U \in \mathcal{U}_k} U,
\]
and
\[
|\mathcal{U}_k| \leq (n+1) \biggl( \left\lceil \frac{1}{\epsilon} \right\rceil +1 \biggr).
\]
Let $\mathcal{U} = \cup_k \mathcal{U}_k$. Note that
\[
\mu\biggl( \bigcup_{U \in \mathcal{U}} \biggr) \geq \mu \biggl( \bigcup_{k=1}^{K} \phi^{-1}([w_k]) \biggr) = \nu \biggl( \bigcup_{k=1}^K [w_k] \biggr) \geq \alpha,
\]
and furthermore
\[
|\mathcal{U}| \leq \sum_{k=1}^K |\mathcal{U}_k| \leq (n+1)  \biggl( \left\lceil \frac{1}{\epsilon} \right\rceil +1 \biggr) r(n,\alpha).
\]
Hence
\[
s(n,\epsilon,\alpha) \leq |\mathcal{U}| \leq (n+1)  \biggl( \left\lceil \frac{1}{\epsilon} \right\rceil +1 \biggr) r(n,\alpha).
\]
Taking the limit supremum of this inequality as $n$ tends to infinity yields
\[
 \limsup_n \frac{1}{n} \log s(n,\epsilon,\alpha) 
 \leq  \limsup_n \frac{1}{n} \log r(n,\alpha) .
\]
By Katok \cite{Katok1980}, as we let $\epsilon$ tend to $0$, we obtain
\[
h(\mu) \leq  h(\nu).
\]
\end{PfofPropPhiPreservesEntropy}


\end{document}